\documentclass[11pt,reqno]{amsart}
\usepackage[utf8]{inputenc}  
\usepackage[T1]{fontenc}
\usepackage{graphicx}
\usepackage{cite}
\usepackage{amssymb}
\usepackage[normalem]{ulem}

\usepackage{fixltx2e}
\usepackage{newunicodechar}
\usepackage{color,xcolor}
\usepackage[normalem]{ulem}

\usepackage{cleveref}

\newtheorem{thm}{Theorem}[section]
\newtheorem{lem}[thm]{Lemma}

\newtheorem{prop}[thm]{Proposition}
\newtheorem{rems}[thm]{Remarks}
\newtheorem{rem}[thm]{Remark}

\theoremstyle{definition}

\numberwithin{equation}{section}

\DeclareMathAlphabet{\mathpzc}{OT1}{pzc}{m}{it}

\newcommand{\N}{\mathbb N}
\newcommand{\R}{\mathbb R}
\newcommand{\C}{\mathbb C}
\newcommand{\K}{\mathbb K}

\newcommand{\B}{\mathbb B}

\newcommand{\kL}{\mathcal L}
\newcommand{\ve}{\varepsilon}

\newcommand{\wh}{\widehat}
\newcommand{\ov}{\overline}
\newcommand{\rd}{\mathrm{d}}

\begin{document}

 \author{Bogdan--Vasile Matioc}
\address{Fakult\"at f\"ur Mathematik, Universit\"at Regensburg,   93053 Regensburg, Germany.}
\email{bogdan.matioc@ur.de}

\author{Christoph Walker}
\address{Leibniz Universit\"at Hannover,
Institut f\"ur Angewandte Mathematik,
Welfengarten 1,
30167 Hannover,
Germany.}
\email{walker@ifam.uni-hannover.de}

\title[Well-Posedness in Time-Weighted Spaces]{Well-Posedness of Quasilinear Parabolic Equations in Time-Weighted Spaces}

\begin{abstract}
Well-posedness in time-weighted spaces of certain  quasilinear (and semilinear) parabolic evolution equations $u'=A(u)u+f(u)$ is established. 
The focus lies on the case of strict inclusions $\mathrm{dom}(f)\subsetneq \mathrm{dom}(A)$ of the domains of the nonlinearities $u\mapsto f(u)$ and  $u\mapsto A(u)$. 
Based on regularizing effects of parabolic equations it is shown that a semiflow is generated in intermediate spaces. 
In applications this allows one to derive global existence from weaker a priori estimates.  
The result is illustrated by  examples of chemotaxis systems.
\end{abstract}

\keywords{Quasilinear parabolic problem; semiflow; well-posedness; chemotaxis equations}
\subjclass[2020]{35B40; 35K58; 35K59; 35Q92} 

\maketitle
%%%%%%%%%%%%%%%%%%%%%%%%%%%%%%%%%%%%%%%%%
  %%%%%%%%%%%%%%%%%%%%%%%%%%%%%%%%%%%%%%%%%
  %%%%%%%%%%%%%%%%%%%%%%%%%%%%%%%%%%%%%%%%%
%%%%%%%%%%%%%%%%%%%%%%%%%%%%%%%%%%%%
 %%%%%%%%%%%%%%%%%%%%%%%%%%%%%%%%%%%
 %%%%%%%%%%%%%%%%%%%%%%%%%%%%%%%%%%
\section{Introduction}\label{Sec:1}
%%%%%%%%%%%%%%%%%%%%%%%%%%%%%%%%%%%%%%%%%
  %%%%%%%%%%%%%%%%%%%%%%%%%%%%%%%%%%%%%%%%%
  %%%%%%%%%%%%%%%%%%%%%%%%%%%%%%%%%%%%%%%%%
%%%%%%%%%%%%%%%%%%%%%%%%%%%%%%%%%%%%
 %%%%%%%%%%%%%%%%%%%%%%%%%%%%%%%%%%%
 %%%%%%%%%%%%%%%%%%%%%%%%%%%%%%%%%%

Let $E_0$ and $E_1$ be two Banach spaces  over $\K\in \{\R,\C\}$ with continuous and dense embedding 
$$
E_1 \stackrel{d}{\hookrightarrow} E_0\,.
$$
 For each $\theta\in (0,1)$, let $(\cdot,\cdot)_\theta$ be an arbitrary admissible interpolation functor of 
 exponent $\theta$ and denote by ${E_\theta:= (E_0,E_1)_\theta}$ the corresponding Banach space with norm $\|\cdot\|_\theta$. Then 
$$
E_1\stackrel{d}{\hookrightarrow} E_\theta\stackrel{d}{\hookrightarrow} E_\vartheta \stackrel{d}{\hookrightarrow} E_0\,,\qquad 0\le \vartheta\le \theta\le 1\,.
$$
 In this paper we shall focus our attention on quasilinear parabolic problems
\begin{equation}\label{EE}
u'=A(u)u+f(u)\,,\quad t>0\,,\qquad u(0)=u^0\,.
\end{equation} 
Here, we assume for the quasilinear part
\begin{subequations}\label{ASS}
\begin{equation}\label{a1c}
A\in C^{1-}\big(O_\beta,\mathcal{H}(E_1,E_0) \big)\,,
\end{equation}
where
\begin{equation}\label{a1b}
\beta\in [0,1)\quad \text{ and  }\quad  \emptyset\not= O_\beta\, \text{ is an open subset of } E_\beta\,,
\end{equation}
 and where $\mathcal{H}(E_1,E_0)$ is the open subset of the bounded linear operators $\kL(E_1,E_0)$  
consisting of generators of strongly continuous analytic semigroups on $E_0$.
 Moreover, for the semilinear part we assume that there are numbers
\begin{equation}\label{a1a}
0<\gamma<1\,,\qquad \beta\le \xi <1\,, \qquad q\ge 1\,,
\end{equation}
such that $f:O_\xi\to E_\gamma$ is locally Lipschitz continuous on the  open subset $O_\xi:=O_\beta\cap E_\xi$ of $E_\xi$ in the sense that for each $R>0$ there is $c(R)>0$  such that,
 for all $w,\,v\in O_\xi\cap \bar{\mathbb{B}}_{E_\beta}(0,R)$,
\begin{equation}\label{a1d}
\|f(w)-f(v)\|_{\gamma}\le  c(R)\big[1+\|w\|_{\xi}^{q-1}+\|v\|_{\xi}^{q-1}\big]\big[\big(1+\|w\|_{\xi}+\|v\|_{\xi}\big) \|w-v\|_{\beta}+\|w-v\|_{\xi}\big]\,.\\
\end{equation}
As for the initial value we fix
\begin{equation}\label{a1e}
\alpha\in (\beta,1) \quad \text{ with }\quad  (\xi-\alpha)q<\min\{1,1+\gamma-\alpha\}\,.
\end{equation}
\end{subequations}
Under assumptions~\eqref{ASS} we shall prove  that problem~\eqref{EE} is well-posed  for  any $u^0\in O_\alpha:=O_\beta\cap E_\alpha$ 
and in fact generates a semiflow on  $O_\alpha$. Note that~\eqref{a1d} includes in particular the case when
\begin{equation}\label{CPW}
\|f(w)-f(v)\|_{\gamma}\le  c(R)\big(1+\|w\|_{\xi}^{q-1}+\|v\|_{\xi}^{q-1}\big)\|w-v\|_{\xi}\,,\qquad
w,v\in O_\xi\cap \bar{\mathbb{B}}_{E_\beta}(0,R)\,,
\end{equation} 
and that \eqref{a1e} is satisfied if $\alpha\in (\xi,1)$.
 The local Lipshitz continuity property~\eqref{a1d} and its stronger version  \eqref{CPW} appear quite naturally in applications, 
see Lemma~\ref{g-Lemma} and the examples in Section~\ref{Sec4} and Section~\ref{Sec6}.\\

Of course, well-posedness of quasilinear and even fully nonlinear equations is well-established, see e.g. 
\cite{Am88,Amann_Teubner,ClementLi,CS01,G88,Lu84,Lu85,Lu85b,P02,PS16,PSW18,MW_MOFM20} for the former and e.g. \cite{An90,AT88,DaPL88,DaP96,DaPG79,Lu87,L95} for the latter problems. 
In particular, a general result on existence of solutions to~\eqref{EE} is stated in \cite[Theorem~12.1]{Amann_Teubner} 
(and established in \cite{Am88}, see also \cite{MW_MOFM20}) for the case that the nonlinearities  $A$ and $f$
are defined and Lipschitz continuous on the same set $O_\beta$. More precisely, it is proven therein that if
\begin{equation*}
(A,f)\in C^{1-}\big(O_\beta,\mathcal{H}(E_1,E_0)\times E_\gamma\big)\,,\qquad u^0\in O_\alpha\,,\qquad 0<\gamma\le \beta<\alpha<1\,,
\end{equation*}
(which is a special case of~\eqref{ASS} taking $q=1$ and~$\xi=\beta<\alpha$), then problem~\eqref{EE} has    a unique maximal strong solution
\begin{equation*} 
 u=u(\cdot;u^0)\in C^1((0,t^+(u^0)),E_0)\cap C((0,t^+(u^0)),E_1)\cap  C ([0,t^+(u^0)),O_\alpha)\cap C^{ \alpha-\theta}\big([0,t^+(u^0)), E_\theta\big)\,
\end{equation*}
for $\theta\in [0,\alpha]$. 
 Moreover, the map $(t,u^0)\mapsto u(t;u^0)$ is a semiflow on $O_\alpha$, and hence $t^+(u^0)=\infty$  if the corresponding orbit is relatively compact in $O_\alpha$.

Herein we shall prove with Theorem~\ref{T1} below a similar result to \cite[Theorem~12.1]{Amann_Teubner}  (see also \cite{Am88}) 
for problem~\eqref{EE}, but under the more general assumptions~\eqref{ASS}.
 In fact, as pointed out above, the interesting and new case occurs when~${\beta<\alpha<\xi}$ in~\eqref{ASS}. 
 Note that then $E_\xi\hookrightarrow E_\alpha\hookrightarrow E_\beta$ and hence, the semilinear part  $f$, being defined   on $  O_\xi$,  
 needs not be defined on the phase space $E_\alpha$ and  requires possibly more regularity than the quasilinear part  $A$.
It is worth emphasizing that also for this case we establish that problem~\eqref{EE} induces a semiflow on $O_\alpha$ 
 so that relatively compact orbits in $O_\alpha$ are global. The global existence criterion can thus be stated in weaker norms than e.g. in \cite{Amann_Teubner}. 

 For the proof we  rely on regularizing effects for quasilinear parabolic equations and work in time-weighted spaces $C_\mu(0,T],E_\xi)$  of continuous functions~${v:(0,T]\to E_\xi}$ 
satisfying \mbox{$\lim_{t\to 0}t^\mu\|v(t)\|_\xi=0$}, where $T>0 $ and  $\mu>(\xi-\alpha)_+$ with $x_+:=\max\{0,\,x\}$ for $x\in\R$.

Time-weighted spaces were used previously for quasilinear evolution problems in the context of maximal regularity in \cite{An90,CS01} and  later in \cite{CPW10, KPW10, PW17, PS16}. 
In particular,  well-posedness of \eqref{EE} is established in \cite{CPW10} in time-weighted $L_p$-spaces assuming that~$f$ satisfies \eqref{CPW} 
along with inequality~\eqref{a1e} for $\gamma=0$ in the scale of real interpolation spaces and 
assuming that the operator $A(u)$ has the property of maximal \mbox{$L_p$-regularity} for~${u\in O_\alpha}$ (see \cite{PW17} for an improvement with equality in~\eqref{a1e} when $E_0$ is a UMD space). 
Furthermore, we refer to \cite{LeCroneSimonett20} for a result in the same spirit based on the concept of
continuous maximal regularity in time-weighted spaces and assuming \eqref{CPW} 
 with inequality~\eqref{a1e} in the scale of continuous interpolation spaces.
Theorem~\ref{T1} below is a comparable result outside the setting of maximal regularity and for arbitrary (admissible) interpolation functors under the more general version~\eqref{a1d} of \eqref{CPW}.

Finally, we point out that, in order to impose weaker conditions on the initial values, time-weighted spaces  $C_\mu(0,T],E)$ (with suitable Banach spaces $E$) were used for concrete semilinear problems (see \eqref{SC1SEEE} below) with bilinear right-hand sides (i.e.~$q=2$ in \eqref{CPW}) even before \cite{A00,AW05,WW06} and recently \cite{LW22}. Theorem~\ref{XT:1} below provides a general result for this case, thereby sharpening the result for the quasilinear problem.\\

Our first main result is Theorem~\ref{T1} and  establishes  the well-posedness of the quasilinear evolution problem \eqref{EE} restricted to the assumptions~\eqref{ASS}.

%\newpage
%%%%%%%%%%%%%%%%%%%%%%%%%%%%%%%%%%%%%%%%%%%%%%%%%%%%%%%%%%%%%%%%%%%%
%%%%%%%%%%%%%%%%%%%%%%%%%%%%%%%%%%%%%%%%%%%%%%%%%%%%%%%%%%%%%%%%%%%%
%%%%%%%%%%%%%%%%%%%%%%%%%%%%%%%%%%%%%%%%%%%%%%%%%%%%%%%%%%%%%%%%%%%%
\begin{thm}\label{T1}
Suppose \eqref{ASS}.

\begin{itemize}
\item[(i)] \emph{(Existence)} Given any $u^0\in O_\alpha$, the Cauchy problem \eqref{EE} possesses a  maximal strong  solution
\begin{equation*} 
\begin{aligned}
 u(\cdot;u^0)&\in C^1\big((0,t^+(u^0)),E_0\big)\cap C\big((0,t^+(u^0)),E_1\big)\cap  C \big([0,t^+(u^0)),O_\alpha\big)
\end{aligned}
\end{equation*}
with $ t^+(u^0)\in(0,\infty]$. Moreover, 
\[
u(\cdot;u^0)\in C^{\min\{\alpha-\theta,\,(1-\mu q)_+\}}\big([0, T], E_\theta\big)\cap   C_{ \mu}\big((0,T], E_\xi\big)
\]  
for all $T<t^+(u^0)$,  where $\theta\in [0,\alpha]$ and $\mu>(\xi-\alpha)_+$. \vspace{1mm}

\item[(ii)] {\em (Uniqueness)} If  $$\tilde u \in C^1\big((0,T],E_0\big)\cap C\big((0,T],E_1\big)\cap C^\vartheta \big([0,T],O_\beta\big) \cap C_\nu\big((0,T],E_\xi\big)$$ 
 is a solution to  \eqref{EE} for some~$T>0$, $\vartheta\in(0,1)$, and $\nu\ge 0$ with $q\nu<\min\{ 1,1+\gamma-\alpha\}$, then $T<t^+(u^0)$ and $\tilde u=u(\cdot;u^0)$ on $[0,T].$ \vspace{1mm}

\item[(iii)] {\em (Continuous dependence)} The map $(t,u^0)\mapsto u(t;u^0)$ is a semiflow on $O_\alpha$. \vspace{1mm}

\item[(iv)] {\em (Global existence)} If the orbit $u([0,t^+(u^0));u^0)$ is relatively compact in $O_\alpha$, then~${t^+(u^0)=\infty}$.\vspace{1mm} 

\item[(v)] {\em (Blow-up criterion)} Let $u^0\in O_\alpha$ be such that $t^+(u^0)<\infty$. \vspace{1mm}
\begin{itemize}
\item[($a$)] If
$u(\cdot;u^0):[0,t^+(u^0))\to E_\alpha$ is uniformly continuous, then
\begin{equation}\label{1a} 
 \lim_{t\nearrow t^+(u^0)}\mathrm{dist}_{E_\alpha}\big( u(t;u^0),\partial O_\alpha\big)=0\,.
\end{equation}

\item[($b$)]  If $E_1$ is compactly embedded in $E_0$, then 
\begin{equation}\label{1}
\lim_{t\nearrow t^+(u^0)}\|u(t;u^0)\|_{\theta}=\infty\qquad\text{ or }\qquad \lim_{t\nearrow t^+(u^0)}\mathrm{dist}_{E_\beta}\big( u([0,t];u^0),\partial O_\beta\big)=0
\end{equation}
for each $\theta\in (\beta,1)$ with $(\xi-\theta)q<\min\{1,1+\gamma-\theta\}$.
\end{itemize}
\end{itemize}
\end{thm}

 Criterion~(iv) yields global existence when the orbit is relatively compact in $E_\alpha$. 
  In particular, if~$E_1$ embeds compactly in $E_0$, 
then a priori estimates on the solution in $E_\alpha$ are sufficient for global existence as noted in condition~\eqref{1} (in contrast  e.g. to \cite{Amann_Teubner} 
where estimates in $E_\xi$ would be needed for the same conclusion).

The proof of Theorem~\ref{T1} relies on a classical fixed point argument. 
However, the technical details do not seem to be completely straightforward due to the singularity of  $t\mapsto f(u(t))$ at $t=0$ which has to be monitored carefully.

%%%%%%%%%%%%%%%%%%%%%%%%%%%%%%%%%%%%%%%%%%%%%%%%%%%%%%%%%%%%%%%%%%%%
%%%%%%%%%%%%%%%%%%%%%%%%%%%%%%%%%%%%%%%%%%%%%%%%%%%%%%%%%%%%%%%%%%%%
%%%%%%%%%%%%%%%%%%%%%%%%%%%%%%%%%%%%%%%%%%%%%%%%%%%%%%%%%%%%%%%%%%%%

\subsection*{ Semilinear Parabolic Problems}

Of course, the result for the quasilinear case remains true for semilinear  parabolic equations 
\begin{equation*}
u'=Au+f(u)\,,\quad t>0\,,\qquad u(0)=u^0\,,
\end{equation*} 
or, more generally,  for  parabolic evolution equations
\begin{equation}\label{SC1SEEE}
u'=A(t)u+f(u)\,,\quad t>0\,,\qquad u(0)=u^0\,,
\end{equation}
with time-dependent operators $A=A(t)$.
 In this setting Theorem~\ref{T1} can be sharpened though.
We present  with Theorem~\ref{XT:1} below a result for the particular case that $f$ is defined on the whole interpolation space $E_\xi$. More precisely, let
\begin{subequations}\label{V16}
\begin{equation}
A\in C^\rho(\R^+,\mathcal{H}(E_1,E_0))
\end{equation} 
for some $\rho>0$ and let 
\begin{equation}\label{16b}
  0\le \alpha\leq  \xi\le 1\,,\quad 0\le \gamma <1\,,\quad (\gamma,\xi)\not= (0,1)\,,\quad q\ge 1\,,\quad (\xi-\alpha)q<\min\{1,1+\gamma-\alpha\}\,.
\end{equation} 
Assume that $f:E_\xi\to E_\gamma$ is locally Lipschitz continuous  in the sense that for each $R>0$ there is a constant~${c(R)>0}$  such that
\begin{equation}\label{Xa1d}
\|f(w)-f(v)\|_{\gamma}\le  c(R)\big[1+\|w\|_{\xi}^{q-1}+\|v\|_{\xi}^{q-1}\big]\big[\big(1+\|w\|_{\xi}+\|v\|_{\xi}\big) \|w-v\|_{{ \alpha}}+\|w-v\|_{\xi}\big]
\end{equation}
for all $w,\,v\in E_\xi\cap \bar{\mathbb{B}}_{E_{ \alpha}}(0,R)$. 
\end{subequations}

It is worth pointing out that
 we may choose the  phase space of the evolution as well as the  target space of the semilinearity $f$  as $E_0$  (that is, we may set $\alpha=\gamma=0$) and that the nonlinearity~$f(u)$
  need not be defined on the phase space $E_\alpha$; see also Remark~\ref{Rem1} below for more details. 
   The well-posedness result regarding the semilinear problem~\eqref{SC1SEEE} under assumption~\eqref{V16} reads as follows:

\begin{thm}\label{XT:1}
 Suppose~\eqref{V16}. 

\begin{itemize}
\item[(i)] (Existence) Given any $u^0\in E_\alpha$, the Cauchy problem \eqref{SC1SEEE} possesses a  maximal strong  solution
\begin{equation*} 
\begin{aligned}
 u(\cdot;u^0)&\in C^1\big((0,t^+(u^0)),E_0\big)\cap C\big((0,t^+(u^0)),E_1\big)\cap  C \big([0,t^+(u^0)),E_\alpha\big) 
\end{aligned}
\end{equation*}
with $ t^+(u^0)\in(0,\infty]$.  Moreover, 
\begin{equation*} 
 u(\cdot;u^0)\in  C^{ \min\{\alpha-\theta,\,(1-\mu q)_+\}}\big([0,T], E_\theta\big)\cap  C_{ \mu}\big((0,T], E_\xi\big)
\end{equation*}
for  all $\theta\in [0,\alpha]$, $\mu>\xi-\alpha$, and $T<t^+(u^0)$. \vspace{1mm}

 \item[(ii)] {\em (Blow-up criterion)} If $u^0\in E_\alpha$ is such that $t^+(u^0)<\infty$, then
\begin{equation}\label{X1}
 \underset{t\nearrow t^+(u^0)}\limsup\|u(t;u^0)\|_{\alpha}=\infty\,.
\end{equation}

\item[(iii)]  {\em (Uniqueness)} If  $$\tilde u \in C^1\big((0,T],E_0\big)\cap C\big((0,T],E_1\big)\cap C  \big([0,T],E_{ \alpha}\big) \cap C_\nu\big((0,T],E_\xi\big)$$ 
 is a solution to  \eqref{SC1SEEE} for some~$T>0$ and    $\nu\geq 0$ with $q\nu<\min\{ 1,1+\gamma-\alpha\}$, then $T<t^+(u^0)$ and $\tilde u=u(\cdot;u^0)$ on $[0,T].$ 
\end{itemize}

Moreover, if $A(t)=A\in \mathcal{H}(E_1,E_0)$  for all $t\geq 0$, then:  
\begin{itemize}
\item[(iv)] {\em (Continuous dependence)} The map $(t,u^0)\mapsto u(t;u^0)$ is a semiflow on $E_\alpha$. \vspace{1mm}

\item[(v)] {\em (Global existence)} If the orbit $u([0,t^+(u^0));u^0)$ is relatively compact in $E_\alpha$, then~${t^+(u^0)=\infty}$.\vspace{1mm} 
\end{itemize}
\end{thm}

  Note that an priori bound in $E_\alpha$ already ensures that the solution is globally defined even 
  in the case of a  non-compact  embedding~${E_1\hookrightarrow E_0}$. \\

%%%%%%%%%%%%%%%%%%%%%%%%%%%%%%%%%%%%%%%%%%%%%%%%%%%%%%%%%%%%%%%%%%%%
%%%%%%%%%%%%%%%%%%%%%%%%%%%%%%%%%%%%%%%%%%%%%%%%%%%%%%%%%%%%%%%%%%%%
%%%%%%%%%%%%%%%%%%%%%%%%%%%%%%%%%%%%%%%%%%%%%%%%%%%%%%%%%%%%%%%%%%%%

The proof of Theorem~\ref{T1} is presented in Section~\ref{Sec2},  while Theorem~\ref{XT:1} is established  in Section~\ref{Sec2x}.
To prepare applications  of these results we state some auxiliary results in  Section~\ref{Sec3}. 
In the subsequent Section~\ref{Sec4} and Section~\ref{Sec6}  we will then provide some applications of Theorem~\ref{T1} and Theorem~\ref{XT:1} 
  to certain chemotaxis systems   featuring cross-diffusion terms,
in particular with focus on the global existence criterion.

%%%%%%%%%%%%%%%%%%%%%%%%%%%%%%%%%%%%%%%%%%%%%%%%%%%%%%%%%%%%%%%%%%%%
%%%%%%%%%%%%%%%%%%%%%%%%%%%%%%%%%%%%%%%%%%%%%%%%%%%%%%%%%%%%%%%%%%%%
%%%%%%%%%%%%%%%%%%%%%%%%%%%%%%%%%%%%%%%%%%%%%%%%%%%%%%%%%%%%%%%%%%%%

%%%%%%%%%%%%%%%%%%%%%%%%%%%%%%%%%%%%%%%%%%%%%%%%%%%%%%%%%%%%%%%%%%%%
%%%%%%%%%%%%%%%%%%%%%%%%%%%%%%%%%%%%%%%%%%%%%%%%%%%%%%%%%%%%%%%%%%%%
%%%%%%%%%%%%%%%%%%%%%%%%%%%%%%%%%%%%%%%%%%%%%%%%%%%%%%%%%%%%%%%%%%%%
\section{Proof of Theorem~\ref{T1}}\label{Sec2}
%%%%%%%%%%%%%%%%%%%%%%%%%%%%%%%%%%%%%%%%%%%%%%%%%%%%%%%%%%%%%%%%%%%%
%%%%%%%%%%%%%%%%%%%%%%%%%%%%%%%%%%%%%%%%%%%%%%%%%%%%%%%%%%%%%%%%%%%%
%%%%%%%%%%%%%%%%%%%%%%%%%%%%%%%%%%%%%%%%%%%%%%%%%%%%%%%%%%%%%%%%%%%%

The proof of Theorem~\ref{T1} is based on Proposition~\ref{P1} below. Before we address  the latter result, let us first recall some basic facts used in the proofs.

%%%%%%%%%%%%%%%%%%%%%%%%%%%%%%%%%%%%%%%%%%%%%%%%%%%%%%%%%%%%%%%%%%%%
%%%%%%%%%%%%%%%%%%%%%%%%%%%%%%%%%%%%%%%%%%%%%%%%%%%%%%%%%%%%%%%%%%%%
%%%%%%%%%%%%%%%%%%%%%%%%%%%%%%%%%%%%%%%%%%%%%%%%%%%%%%%%%%%%%%%%%%%%
\subsection*{Preliminaries}
%%%%%%%%%%%%%%%%%%%%%%%%%%%%%%%%%%%%%%%%%%%%%%%%%%%%%%%%%%%%%%%%%%%%
%%%%%%%%%%%%%%%%%%%%%%%%%%%%%%%%%%%%%%%%%%%%%%%%%%%%%%%%%%%%%%%%%%%%
%%%%%%%%%%%%%%%%%%%%%%%%%%%%%%%%%%%%%%%%%%%%%%%%%%%%%%%%%%%%%%%%%%%%

Let $T>0$,  $\mu\in \mathbb{R}$, and consider a Banach space $E$. We denote  by $C_{\mu}((0,T],E)$ the Banach space of all functions
$u\in C((0,T],E)$ such that 
$t^{\mu} u(t)\rightarrow 0$ in $E$ as $t\rightarrow 0$,  equipped with the norm
\begin{equation*}
u\mapsto \|u\|_{C_{\mu}((0,T],E)} := \sup\left\{ t^{\mu}\, \|u(t)\|_E \,:\, t\in (0,T]\right\}\,.
\end{equation*}
 Note that 
\begin{equation}\label{Emb}
C_{\mu}((0,T],E)\hookrightarrow C_{\nu}((0,T],E)\,,\quad \mu\le \nu\,.
\end{equation}

Given $\omega>0$ and $\kappa\ge 1$,  we denote by~${\mathcal{H}(E_1,E_0;\kappa,\omega)}$ the class of all operators $\mathcal{A}\in\mathcal{L}(E_1,E_0)$ such 
that $\omega-\mathcal{A}$ is an
isomorphism from $E_1$ onto~$E_0$ and
$$
\frac{1}{\kappa}\,\le\,\frac{\|(\mu-\mathcal{A})z\|_{0}}{\vert\mu\vert \,\| z\|_{0}+\|z\|_{1}}\,\le \, \kappa\ ,\qquad {\rm Re\,}
\mu\ge \omega\ ,\quad z\in E_1\setminus\{0\}\,.
$$
 Then $$\mathcal{H}(E_1,E_0)=\bigcup_{\omega>0\,,\,\kappa\ge 1} \mathcal{H}(E_1,E_0;\kappa,\omega)\,.$$
For time-dependent operators $\mathcal{A}\in C^\rho(I,\mathcal{H}(E_1,E_0))$ with $\rho\in(0,1)$ 
there exists a unique parabolic evolution operator $U_\mathcal{A}(t,s)$, $0\le s\le t<\sup I$, in the sense of \cite[II. Section 2]{LQPP}. 

Based on this, we may reformulate the quasilinear Cauchy problem \eqref{EE}  as a fixed point equation of the form
\begin{equation}\label{ttp}
u(t)=U_{A(u)}(t,0)u^0+ \int_0^t U_{A(u)}(t,\tau) f(u(\tau))\,\rd \tau\,,\quad t>0\,,
\end{equation}
see the proof  below of Proposition~\ref{P1}.

\begin{prop}\label{P1}
Suppose \eqref{ASS}.
Let $S_\alpha \subset O_\alpha$ be a compact subset of $E_\alpha$.
Then, there exist a neighborhood~$Q_\alpha$ of $S_\alpha$ in $O_\alpha$ and~${T:=T(S_\alpha)>0}$ such that, for each $u^0\in Q_\alpha$,
the problem \eqref{EE} has a strong solution 
\begin{equation}\label{regul}
\begin{aligned}
u=u(\cdot;u^0)\in &\ C^1\big((0,T],E_0\big)\cap C\big((0,T],E_1\big)\cap C \big([0,T],O_\alpha\big) \\[1ex]
&\cap  C^{\min\{\alpha-\theta,\,(1-\mu q)_+\}}\big([0,T], E_\theta\big)\cap  C_\mu\big((0,T],E_\xi\big)
\end{aligned}
\end{equation}
for any $\theta\in [0,\alpha]$ and  $\mu>(\xi-\alpha)_+$. 
Moreover, there is a constant $c_0:=c_0(S_\alpha)>0$ such that
$$
\| u(t;u^0)-u(t;u^1)\|_\alpha\le c_0 \| u^0-u^1\|_\alpha\,,\qquad 0\le t\le T\,,\quad u^0, u^1\in Q_\alpha\,.
$$
Finally, if  
\begin{equation}\label{util}
\tilde u\in C^1((0,T],E_0)\cap C((0,T],E_1)\cap C^\vartheta ([0,T],O_\beta) \cap C_\nu\big((0,T],E_\xi\big)
\end{equation}
with  $ \vartheta\in(0,1)$
 and $0\le q\nu<\min \{ 1, 1+\gamma-\alpha \}$ 
 is a solution to  \eqref{EE} satisfying $\tilde u (0)=u^0 \in Q_\alpha$, then $\tilde u=u(\cdot;u^0)$.
\end{prop}

\begin{proof}   We devise the proof into several steps.\medskip

\noindent{\em The fixed point formulation.}
Since   $S_\alpha \subset O_\alpha$  is compact in $E_\alpha$ and since $E_\alpha$ embeds continuously into $E_\beta$, we find a constant~$\delta>0$ 
such that  \mbox{$\mathrm{dist}_{E_\beta}(S_\alpha, \partial O_\beta) > 2\delta > 0$}.
Moreover, due to~\eqref{a1c} and \cite[II.~Proposition~6.4]{A83}, $A$ is uniformly Lipschitz continuous on some neighborhood of $S_\alpha$, hence there are $\ve > 0$ and~$L>0$ such that 
$$
\ov\B_{E_\beta}(S_\alpha, 2\ve) 
\subset \mathbb{B}_{E_\beta}(S_\alpha, \delta) \subset O_\beta
$$ 
and
\begin{equation}\label{e4a}
\|A(x) - A(y)\|_{\mathcal{L}(E_1,E_0)} \le L\|x-y\|_\beta\ , \quad x,y \in 
\ov\B_{E_\beta}(S_\alpha, 2\ve)\ . 
\end{equation}
The compactness of $A(S_\alpha)$  in $\mathcal{H}(E_1,E_0)$ implies according to
 \cite[I.~Corollary 1.3.2]{LQPP} that there are~${\kappa\ge 1}$ and $\omega>0$ such that (making $\ve>0$ smaller, if necessary)  
\begin{equation}\label{e4b}
A(x)\in \mathcal{H}(E_1,E_0;\kappa,\omega)\,,\quad x \in \ov\B_{E_\beta}(S_\alpha, 2\ve)\,.
\end{equation}
Recalling~\eqref{a1e} and \eqref{Emb}, we may choose $\alpha_0\in (\beta,\alpha)$ if $\alpha<\xi$ respectively put $\alpha_0:=\xi$ if $\alpha\ge \xi$, choose $\gamma_0\in (0,\gamma)$, and assume that
$\mu$ satisfies
\begin{equation}\label{much}
\xi-\alpha_0<\mu<\min\left\{\frac{1}{q},\frac{1+\gamma_0-\alpha}{q}\right\}\,.
\end{equation}
Fix $\rho\in (0,\min\{\alpha-\beta,1-\mu q\})$. 
Given $T \in (0,1)$ (chosen small enough as specified later on), define a closed subset of $C([0,T], E_\beta)$ by
$$
\mathcal{V}_T := \left\{v \in C([0,T], E_\beta) \,:\, v(t)\in \ov\B_{E_\beta}(S_\alpha, 2\ve)\,,\,\|v(t) - v(s)\|_\beta \le |t-s|^{\rho}\ , ~0 \le s,t \le T\right\}\,.
$$
Hence, if $v\in\mathcal{V}_T$, then \eqref{e4b} ensures 
\begin{subequations}\label{LLP}
\begin{equation}\label{llp}
A(v(t))\in \mathcal{H}(E_1,E_0;\kappa,\omega)\,,\quad t\in [0,T]\,,
\end{equation}
while \eqref{e4a} implies
\begin{equation}\label{llq}
A(v)\in C^\rho\big([0,T],\mathcal{L}(E_1,E_0)\big)\quad\text{with}\quad \sup_{0\le s<t\le T}
\frac{\|  A(v(t))-  A(v(s))\|_{\mathcal{L}(E_1,E_0)}}{( t-s)^\rho}\le L\,.
\end{equation}
\end{subequations} 
For each $v\in\mathcal{V}_T$, the evolution operator
$$
U_{A(v)}(t,s)\,,\quad 0\le s\le t\le T\,,
$$
is thus well-defined  and \eqref{LLP} guarantees that we are 
in a position to use the results of \cite[II.Section~5]{LQPP}.
In particular, due to \cite[II.Lemma~5.1.3]{LQPP} there exists  $c=c(S_\alpha)>0$   such that
\begin{equation}\label{e6a}
\|U_{A(v)}(t,s)\|_{\mathcal{L}(E_\theta)}+(t-s)^{\theta-\vartheta_0}\,\|U_{A(v)}(t,s)\|_{\mathcal{L}(E_\vartheta,E_\theta)}  \le c  \,, \quad 0\le s\le t\le T\,,
\end{equation}
for $0\le \vartheta_0\le \vartheta\le \theta\le 1$ with $\vartheta_0<\vartheta$ if $0<\vartheta< \theta< 1$. 
In the following, we write ${c=c(S_\alpha)}$ for positive constants depending only on~$S_\alpha$ and~${\alpha,\, \beta,\, \gamma,\, \xi,\, \mu,\alpha_0,\, \gamma_0, \,\varepsilon,\, \delta}$, 
 but neither on $v\in\mathcal{V}_T$ nor on~${T\in (0,1)}$. 

We  introduce the complete metric space
$$
\mathcal{W}_T:=\mathcal{V}_T\cap \bar{\mathbb{B}}_{C_\mu((0,T],E_\xi)}(0,1)
$$
equipped with the metric
$$
d_{\mathcal{W}_T}(u,v):=\|u-v\|_{C([0,T],E_\beta)}+\|u-v\|_{C_\mu((0,T],E_\xi)}\,,\quad u,v\in \mathcal{W}_T\,.
$$
Let $u,v\in \mathcal{W}_T$.
Note that $u(t)\in \ov\B_{E_\beta}(S_\alpha, 2\ve) \subset O_\beta$ for $t\in [0,T]$, while $u(t)\in E_\xi$ for $t\in (0,T]$. 
In particular, $u(t)\in O_\xi$ for $t\in (0,T]$. 
Moreover, there is $c=c(S_\alpha)>0$  such that $\|u(t)\|_\beta\le c$ for~${t\in [0,T]}$ and $\|u(t)\|_\xi \le  t^{-\mu}$ for  $t\in (0,T]$. 
It then follows from~\eqref{a1d}  that, for $0<s,t\le T<1$,
\begin{align}\label{e7}
\begin{aligned}
\|f(u(t))-f(u(s))\|_{\gamma} &\le c \big(t^{-\mu q}+s^{-\mu q}\big) \|u(t)-u(s)\|_{\beta}\\[1ex]
&\quad + c \big(t^{-\mu (q-1)}+s^{-\mu (q-1)}\big)\|u(t)-u(s)\|_{\xi} \,,
\end{aligned}
\end{align}
where $c=c(S_\alpha)>0$ is a  constant.
Similarly, fixing $v^0\in  O_\xi$ arbitrarily, we deduce from \eqref{a1d} that
\begin{align}\label{e8}
\|f(u(t))\|_{\gamma}&\le \|f(u(t))-f(v^0)\|_{\gamma} +\|f(v^0)\|_{\gamma}\le  c t^{-\mu q}\,,\quad t\in (0,T]\,,
\end{align}
for some constant $c=c(S_\alpha)>0$. 
Also note that
\begin{align}\label{e9}
\|f(u(t))-f(v(t))\|_{\gamma}\le c t^{-\mu q} \|u(t)-v(t)\|_{\beta}+ c t^{-\mu (q-1)} \|u(t)-v(t)\|_{\xi} \,,\quad t\in (0,T]\,,
\end{align}
where again $c=c(S_\alpha)>0$.
 Set
$$
Q_\alpha:=\mathbb{B}_{E_\alpha}(S_\alpha,\ve/(1+e_{\alpha,\beta}))\subset O_\alpha\,,
$$ 
where $e_{\alpha,\beta}>0$ is the norm of the embedding $E_\alpha\hookrightarrow E_\beta$.
Given $u^0\in Q_\alpha,$ define
\begin{equation}\label{Lambda}
F(u)(t):= U_{A(u)}(t,0)u^0 +\int_0^t U_{A(u)}(t,\tau) f(u(\tau))\,\rd \tau\,,\quad t\in [0,T]\,,\quad u\in\mathcal{W}_T\,.
\end{equation}
 We shall prove that $F:\mathcal{W}_T\rightarrow \mathcal{W}_T$ is a contraction for $T=T(S_\alpha)\in (0,1)$ small enough.  
 To this end, particular attention has to be paid to the singularity of $t\mapsto f(u(t))$ at $t=0$  when  studying the function~$F(u)$.\medskip

\noindent{\em Continuity in $E_\beta$.} To prove the  continuity in $E_\beta$ we note that, for $0\le s<t\le T$, $u\in\mathcal{W}_T$, and~$\theta\in[0,\alpha]$,
\begin{subequations}\label{gfg}
\begin{align}
\|F(u)(t)-F(u)(s)\|_\theta&\le \|U_{A(u)}(t,0)u^0-U_{A(u)}(s,0)u^0\|_{\theta} \nonumber\\
&\quad+\int_0^s \|U_{A(u)}(t,\tau)-U_{A(u)}(s,\tau)\|_{\mathcal{L}(E_\gamma,E_\theta)} \,\|{ f(u(\tau))}\|_\gamma\,\rd \tau \nonumber\\
&\quad 	+\int_s^t \|U_{A(u)}(t,\tau)\|_{\mathcal{L}(E_\gamma,E_\theta)} \,\|{ f(u(\tau))}\|_\gamma\,\rd \tau=:I_1+I_2+I_3\,.
\end{align}
We first note from \eqref{LLP} and \cite[II.Theorem 5.3.1]{LQPP} (with $f=0$ therein) that there exists $c=c(S_\alpha)>0$ with
\begin{align}\label{gfg1}
I_1 \le c (t-s)^{\alpha-\theta}  \|u^0\|_\alpha\,.
\end{align}
Moreover, since 
$$
\|U_{A(u)}(t,s)-1\|_{\mathcal{L}(E_\alpha,E_\theta)}\le c(t-s)^{\alpha-\theta}
$$
due to  \cite[II.~Theorem~5.3.1]{LQPP} and \eqref{LLP}, we use \eqref{e6a} and \eqref{e8} to derive
\begin{align}
I_2 &\le \int_0^s \|U_{A(u)}(t,s)-1\|_{\mathcal{L}(E_\alpha,E_\theta)} \,\|U_{A(u)}(s,\tau)\|_{\mathcal{L}(E_\gamma,E_\alpha)} \|f(u(\tau))\|_\gamma\,\rd \tau\nonumber\\
&\le c (t-s)^{\alpha-\theta}\,\max\Big\{\int_0^s (s-\tau)^{\gamma_0-\alpha} \tau ^{-\mu q}\,\rd \tau,\, \int_0^s  \tau ^{-\mu q}\,\rd \tau\Big\}\nonumber\\
&\le  c\,\max\{T^{1+\gamma_0-\alpha-\mu q},\,T^{1-\mu q}\}\, (t-s)^{\alpha-\theta}\,, \label{gfg2}
\end{align}
since
\[
\int_0^s (s-\tau)^{\gamma_0-\alpha} \tau ^{-\mu q}\,\rd \tau=s^{1+\gamma_0-\alpha-\mu q}\int_0^1 (1-\tau)^{\gamma_0-\alpha} \tau ^{-\mu q}\,\rd \tau\leq T^{1+\gamma_0-\alpha-\mu q} \mathsf{B}(1+\gamma_0-\alpha,1-\mu q)\,,
\]
where $\mathsf{B}$ denotes the Beta function.
Using again \eqref{e6a} with~\eqref{e8}, we obtain similarly  
\begin{align}
I_3 &\le    c\max\Big\{\int_s^t (t-\tau)^{\gamma_0-\theta}  \,\tau^{-\mu q}\,\rd \tau,\, \int_s^t \tau^{-\mu q}\,\rd \tau\Big\} \nonumber\\
&\le  c\max \big\{(t-s)^{1+\gamma_0-\alpha-\mu q}(t-s)^{\alpha-\theta},\,(t-s)^{1-\mu q}\big\}\,.\label{gfg3}
\end{align}
\end{subequations}
 Due to \eqref{much} and $\rho\in (0,\min\{\alpha-\beta,1-\mu q\})$, we see from \eqref{gfg} with  $\theta=\beta$ that   we may choose the constant~${T=T(S_\alpha)\in (0,1)}$ small enough to get
\begin{align}\label{e12}
\|F(u)(t)-F(u)(s)\|_\beta &\le   (t-s)^{\rho}\,,\quad 0\le s\le t\le T\,,
\end{align}
and, since $F(u)(0)=u^0$,
\begin{equation*}
\begin{split}
\|F(u)(t)-u^0\|_\beta &\le  T^{\rho}\le \ve\,,\quad 0\le t\le T\,.
\end{split}
\end{equation*}
In particular, we infer from  $u^0\in Q_\alpha=\mathbb{B}_{E_\alpha}(S_\alpha,\ve/(1+e_{\alpha,\beta}))$ that
\begin{align}\label{e14}
F(u)(t)\in  \ov\B_{E_\beta}(S_\alpha, 2\ve)\,,\quad 0\le t\le T\,,
\end{align}  
 hence $F(u)\in\mathcal{V}_T.$\medskip

\noindent{\em Continuity in $E_\xi$.} We now prove that  $F(u)\in C((0,T], E_1)$, hence in particular $F(u)\in C((0,T], E_\xi)$.
To this end we fix $\ve\in (0,T)$ and set $u_\ve(t):=u(t+\ve)$ for $t\in [0,T-\ve]$. 
Then, in view of~\eqref{ttp}  we have $u_\ve\in C([0,T-\ve],E_\xi)$ and \eqref{a1d} entails that $f(u_\ve)\in C([0,T-\ve],E_\gamma)$.
If
 \[
 U_{A(u_\varepsilon)}(t,s)=U_{A(u)}(t+\varepsilon,s+\varepsilon)\,,\quad 0\leq s\leq t\leq T-\varepsilon\,,
\]
 denotes the evolution operator associated with $A(u_\ve)$, we infer from the definition of $F(u)$ that
\begin{equation}\label{P1d}
F(u)(t+\ve)=U_{A(u_\ve)}(t,0)F(u)(\ve)+\int_0^t U_{A(u_\ve)}(t,s) f(u_\ve(s))\,\rd s\,,\quad t\in [0,T-\ve]\,.
\end{equation}
Applying   \cite[II.Theorem~1.2.2, II.Remarks~2.1.2 (e)]{LQPP}, 
yields $$F(u)( \ve+\cdot) \in  C\big((0,T-\varepsilon],E_1\big)\cap C^1\big((0,T-\ve],E_0\big)$$ for all $\ve\in(0,T)$, hence 
\begin{equation}\label{regdes}
F(u)\in C\big((0,T],E_1\big)\cap C^1\big((0,T],E_0\big)\,.
\end{equation}
Similarly, we derive from~\eqref{e6a},~\eqref{e8},  and the definition of $\alpha_0$ that
\begin{align*}
\|F(u)(t)\|_\xi &\le \|U_{A(u)}(t,0)\|_{\mathcal{L}(E_\alpha,E_\xi)}\, \|u^0\|_{\alpha}+\int_0^t \|U_{A(u)}(t,\tau)\|_{\mathcal{L}(E_\gamma,E_\xi)} \,\|f(u(\tau))\|_\gamma\,\rd \tau\\
&\le   c\big(  t^{\alpha_0-\xi} + \max\big\{ t^{1+\gamma_0-\xi -\mu q},\, t^{1-\mu q}\big\}\big)
\end{align*}
for $t\in (0,T]$, hence
\begin{align}\label{o1}
t^\mu \|F(u)(t)\|_\xi &\le    c\big( t^{\alpha_0-\xi+\mu} + \max\big\{ t^{1+\gamma_0-\xi -\mu( q-1)},\, t^{1-\mu q}\}\big)\,,\quad t\in (0,T]\,.
\end{align}
 Owing to \eqref{much} (noticing that $\xi +\mu(q-1)<\mu q+\alpha<1+\gamma_0$)
 we may make $T=T(S_\alpha)\in (0,1)$ smaller, if necessary, to obtain
\begin{align}\label{e15}
\|F(u)\|_{C_\mu((0,T],E_\xi)} &\le 1\,.
\end{align}
It now follows from \eqref{e12}, \eqref{e14}, and~\eqref{e15} that $F:\mathcal{W}_T \to \mathcal{W}_T$ provided that $T=T(S_\alpha)\in (0,1)$ is small enough. \medskip

\noindent{\em The contraction property.} It remains to show that $F$ is a contraction. To this end, let $u,v\in \mathcal{W}_T$ and observe from \cite[II.Lemma~5.1.4]{LQPP},~\eqref{LLP}, and~\eqref{e4a} that there is  
$c=c(S_\alpha)>0$ such that
\begin{equation}\label{e23}
(t-\tau)^{\vartheta-\eta}\|U_{A(u)}(t,\tau)-U_{A(v)}(t,\tau)\|_{\mathcal{L}(E_\eta,E_\vartheta)}\le c\,\|u-v\|_{C([0,T],E_\beta)}\,,\quad 0\le \tau<t\le T\,,
\end{equation}
provided that $0\le \vartheta<1$ and $0<\eta\le 1$.  Moreover, in view of~\eqref{e9}, we have 
\begin{align}\label{r20}
\|f(u(t))-f(v(t))\|_{\gamma}\le c\, t^{-\mu q}\, d_{\mathcal{W}_T}(u,v) \,,\quad t\in (0,T]\,.
\end{align}
Letting $\theta\in \{\beta,\,\alpha,\,\xi\}$, we infer from  \eqref{r20}, \eqref{e23}, \eqref{e8}, and \eqref{e6a} that, for $t\in (0,T]$,
\begin{align}
\|F(u)(t)-F(v)(t)\|_\theta &\le \|U_{A(u)}(t,0)-U_{A(v)}(t,0)\|_{\mathcal{L}(E_\alpha,E_\theta)}\,\| u^0\|_{\alpha}\nonumber\\
&\quad+\int_0^t \|U_{A(u)}(t,\tau)-U_{A(v)}(t,\tau)\|_{\mathcal{L}(E_\gamma,E_\theta)} \,\|f(v(\tau))\|_\gamma\,\rd \tau\nonumber\\
&\quad 	+\int_0^t \|U_{A(u)}(t,\tau)\|_{\mathcal{L}(E_\gamma,E_\theta)} \,\|f(u(\tau))-f(v(\tau))\|_\gamma\,\rd \tau\nonumber\\
&\le  c\,\|u-v\|_{C([0,T],E_\beta)}\big(\|u^0\|_\alpha t^{\alpha-\theta}+t^{1+\gamma-\theta-\mu q}\big)\nonumber\\
&\quad +  c\,d_{\mathcal{W}_T}(u,v)\, \max\big\{t^{1+\gamma_0-\theta-\mu q}, t^{1-\mu q}\}\,.\label{c111}
\end{align}
Taking $\theta=\beta$ and $\theta=\xi$ in \eqref{c111} implies that
\begin{align*}
d_{\mathcal{W}_T}\big(F(u),F(v)\big) &\le  c\, \big( T^{\alpha-\beta}+T^{1+\gamma_0-\beta-\mu q} +  T^{1-\mu q}+T^{\mu+\alpha-\xi}
+T^{1+\gamma_0-\xi-\mu (q-1)}\big) \,d_{\mathcal{W}_T}(u,v)\,.
\end{align*}
Owing to \eqref{ASS} and \eqref{much}, we may make $T=T(S_\alpha)\in (0,1)$ smaller, if necessary, to obtain that
\begin{align*}
d_{\mathcal{W}_T}(F(u),F(v)) &\le  \frac{1}{2} \,d_{\mathcal{W}_T}(u,v)\,,\quad u,v\in \mathcal{W}_T\,.
\end{align*}
This shows that $F:\mathcal{W}_T \to \mathcal{W}_T$ is a contraction for $T=T(S_\alpha)\in (0,1)$ small enough  and thus has 
a unique fixed point $u=u(\cdot;u^0)\in \mathcal{W}_T$ according to Banach's fixed point theorem.

 Since  the (H\"older) continuity  property
 $u\in C^{\min\{\alpha-\theta,1-\mu q\}}\big([0,T], E_\theta\big)$ is established in \eqref{gfg}  for~${\theta\in[0,\alpha)}$, respectively in~\eqref{gfg} and~\eqref{regdes} for $\theta=\alpha$, 
 it follows from~\eqref{regdes}  that $u$ enjoys the regularity properties~\eqref{regul}
  for   $\mu$ as chosen in \eqref{much} (and, in view of~\eqref{Emb}, also for larger values of~$\mu$).
  The arguments leading to~\eqref{regdes} imply also that $u$ is a strong solution to~\eqref{EE}.
  That the regularity properties~\eqref{regdes} hold for every $\mu >(\xi-\alpha)_+$  may be shown by arguing in a similar manner as   below where the uniqueness claim is established. \medskip

\noindent{\em Lipschitz continuity w.r.t. initial data and uniqueness.} To establish the Lipschitz continuity of the solution with respect to the initial values, let ~${u^0, u^1\in Q_\alpha}$.
 We note for $\theta\in \{\beta,\,\alpha,\,\xi\}$ and $t\in (0,T]$ that 
\begin{align*}
\|u(t;u^0)-u(t;u^1)\|_\theta\leq\|U_{A(u(\cdot;u^1))}(t,0)\|_{\mathcal{L}(E_\alpha,E_\theta)}\|u^0-u^1\|_{\alpha} 
 +\|F(u(\cdot;u^0))(t)-F(u(\cdot;u^1))(t)\|_\theta\,.
\end{align*}
In view of \eqref{e6a}, we get
\[
\|U_{A(u(\cdot;u^1))}(t,0)\|_{\mathcal{L}(E_\alpha,E_\theta)}\leq c\left\{
\begin{array}{lll}
1&,& \theta\in\{\beta,\,\alpha\},\\
 t^{\alpha_0-\theta}&,&\theta=\xi\,,
\end{array}
\right.
\]
while \eqref{c111} yields
\begin{align*}
\| F(u(\cdot;u^0))(t)-F(u(\cdot;u^1))(t)\|_\theta &\le  c  \|u(\cdot;u^0)-u(\cdot;u^1)\|_{C([0,T],E_\beta)} \big(t^{\alpha-\theta}+t^{1+\gamma-\theta-\mu q}\big)  \\
&\quad +  c\,d_{\mathcal{W}_T}(u(\cdot;u^0),u(\cdot;u^1))\,  \max\big\{t^{1+\gamma_0-\theta-\mu q}, t^{1-\mu q}\}  \,. 
\end{align*}
Hence, taking first $\theta=\beta$ and $\theta=\xi$  and making $T=T(S_\alpha)\in (0,1)$ smaller, if necessary, we derive
$$
d_{\mathcal{W}_T}\big(u(\cdot;u^0),u(\cdot;u^1)\big)\le c\,  \|u^0-u^1\|_{\alpha}
$$
and then, with $\theta=\alpha$, we deduce that indeed
$$
\|u(t;u^0)-u(t;u^1)\|_\alpha \le c_0\,  \|u^0-u^1\|_{\alpha}\,,\quad u^0, u^1\in Q_\alpha\,,
$$
for some constant $c_0=c_0(S_\alpha)>0$.

Concerning the uniqueness claim, let $\tilde u$ be a solution to \eqref{EE} with initial data $u^0\in Q_\alpha$ with regularity stated in~\eqref{util}. 
We may assume that $q(\xi-\alpha)_+<\nu<\min\{1,1+\gamma-\alpha\}$ due to \eqref{Emb}. Let $u=u(\cdot;u^0)$.
 We choose 
\[
\wh \rho\in\big(0,\min\{\rho,\, \vartheta\}\big)\qquad\text{and}\qquad \wh \mu\in\Big(\max\{\mu,\, \nu\},\,\frac{1+\gamma-\alpha}{q}\Big)
\]
and note that, if $T$ is sufficiently small, then both functions $u$ and  $\tilde u$ belong to the complete metric space $\mathcal{W}_T$ (with $(\rho,\mu)$ replaced by $(\wh \rho,\wh \mu)$).
The uniqueness claim follows now by arguing as in the first part of the proof where the existence of a solution was established.
\end{proof}

%%%%%%%%%%%%%%%%%%%%%%%%%%%%%%%%%%%%%%%%%%%%%%%%%%%%%%%%%%%%%%%%%%%%%%%%%%%%%%%%%%%%%%%%%%%%%%%%%%%%%%%%%%%%%%%%%%%%%%%%%%%%%%%%%%%%%%%%%%%%%%%%%%%%%%%%%%%%%%%%%%%%%%%%%%%%%%%%%%%%%%%%%%%%%%%%%%%%%%%%%%%%%%%%%%%%%%%%%%%%%%%%%%%%%%%%%%%%

\begin{rem}\label{Rem1}
An inspection of the proof of Proposition~\ref{P1} shows that
the H\"older continuity in time and the assumption~$\alpha>\beta$ are only needed to ensure~\eqref{llq} while the assumption~${\gamma>0}$  and~${\xi<1}$ is
  only used when applying formula~\eqref{e23} to derive~\eqref{c111}. 
  That is, these assumptions are required to handle the quasilinear part and can thus be weakened for
   the semilinear problem~\eqref{SC1SEEE}; see the subsequent proof of Proposition~\ref{XP1}.
\end{rem}

%%%%%%%%%%%%%%%%%%%%%%%%%%%%%%%%%%%%%%%%%%%%%%%%%%%%%%%%%%%%%%%%%%%%
%%%%%%%%%%%%%%%%%%%%%%%%%%%%%%%%%%%%%%%%%%%%%%%%%%%%%%%%%%%%%%%%%%%%
%%%%%%%%%%%%%%%%%%%%%%%%%%%%%%%%%%%%%%%%%%%%%%%%%%%%%%%%%%%%%%%%%%%%

\subsection*{Proof of Theorem~\ref{T1}} 

%%%%%%%%%%%%%%%%%%%%%%%%%%%%%%%%%%%%%%%%%%%%%%%%%%%%%%%%%%%%%%%%%%%%
%%%%%%%%%%%%%%%%%%%%%%%%%%%%%%%%%%%%%%%%%%%%%%%%%%%%%%%%%%%%%%%%%%%%
%%%%%%%%%%%%%%%%%%%%%%%%%%%%%%%%%%%%%%%%%%%%%%%%%%%%%%%%%%%%%%%%%%%%

{\bf (i), (ii) Existence and Uniqueness:}
Due to Proposition~\ref{P1}, the Cauchy problem \eqref{EE}
admits for each  $u^0\in O_\alpha=O_\beta\cap E_\alpha$  a unique local strong solution. By standard arguments it can be extended to 
a maximal strong solution $u(\cdot;u^0)$ on the maximal interval of existence $[0,t^+(u^0))$.
 The regularity properties of $u(\cdot;u^0)$  
as stated in part (i) of Theorem~\ref{T1} and the uniqueness claim stated  in part (ii) also follow from Proposition~\ref{P1}.\\

\noindent{\bf (iii) Continuous dependence:} Let $u^0\in O_\alpha$ and choose $t_0\in (0,t^+(u^0))$ arbitrarily.
 Fixing~${t_*\in (t_0,t^+(u^0))}$, the set $S_\alpha:=u([0,t_*];u^0)\subset O_\alpha$ is compact. Thus, we infer from Proposition~\ref{P1} that
 there are $\ve=\ve(S_\alpha)>0$, $T=T(S_\alpha)>0$, and $c_0=c_0(S_\alpha)\ge 1$  such that $T<t^+(u^1)$ for any
 $$u^1\in Q_\alpha=\mathbb{B}_{E_\alpha}(S_\alpha,\ve/(1+e_{\alpha,\beta}))$$ and
  \begin{align}\label{FE2}
\|u(t;u^1)-u(t; u^2)\|_{\alpha}\leq c_0 \|u^1- u^2\|_\alpha\,,\quad 0\le t\le T\,,\quad u^1, u^2\in Q_\alpha\,.
\end{align}
Given $N\ge 1$ with $(N-1)T< t_*\le NT$ define by $V_\alpha:=\mathbb{B}_{E_\alpha}(u^0,\ve_0)$ with 
  $\ve_0:=\ve/((1+e_{\alpha,\beta})c_0^{N-1})$ 
 a neighborhood of $u^0$  in $Q_\alpha$. We then claim that
there is $k_0\geq1$ with
 \begin{itemize}
 \item[(1)] $t_*<t^+(u^1)$ for each $u^1\in V_\alpha$,
  \item[(2)]  $\|u(t;u^1)-u(t;u^0)\|_{ \alpha} \leq k_0 \|u^1- u^0\|_\alpha$ for $0\le t\le t_*$ and $u^1\in V_\alpha$.
 \end{itemize}
Indeed, let $u^1\in V_\alpha$. For $t_*\le T$, this is exactly the above statement. 
If otherwise $T<  t_*$, then we have $u(T;u^0)\in S_\alpha$ and the estimate (see \eqref{FE2})
$$
\|u(t;u^1)-u(t; u^0)\|_{\alpha}\leq c_0 \|u^1- u^0\|_\alpha <\frac{\ve}{1+e_{\alpha,\beta}}\,,\quad 0\leq t\leq T\,,
$$
entails $u(T;u^1)\in Q_\alpha$. Thus $T<t^+(u(T;u^i))$ for $i=0,1$, while the uniqueness
of solutions to \eqref{EE} ensures that $u(t;u(T;u^i))=u(t+T;u^i)$, $0\le t\le T$.
Therefore, it follows from \eqref{FE2} that
$$
\|u(t+T;u^1)-u(t+T; u^0)\|_{\alpha}{ \leq c_0\|u(T;u^1)-u(T; u^0)\|_{\alpha}}\leq  c_0^2 \|u^1- u^0\|_\alpha \,,\quad 0\le t\le T\,.
$$
Now, if  $N=2$ we are done. Otherwise we proceed to derive (1) and (2) after finitely many iterations. 
In particular, property (1) implies that $(0,t_*)\times V_\alpha$ is a neighborhood of $(t_0,u^0)$ in 
$$
\mathcal{D}=\{(t,w)\,:\, 0\le t<t^+(w)\,,\,w\in O_\alpha\}\,.
$$ 
This along with (2) implies the solution map defines a semiflow on $O_\alpha$.\\

%%%%%%%%%%%%%%%%%%%%%%%%%%%%%%%%%%%%%%%%%%%%%%%%%%%%%%%%%%%%%%%%%%%%%%%%%%%%%%%%%%%%%%%%%%%%%%%%%%%%%%%%%%%%%%%%%%%%%%%%%%%%%%%%%%%%%%%%%%%%%%%%%%%%%%%%%%%%%%%%%%%%%%%%%%%%%%%%%%%%%%%%%%%%%%%%%%%%%%%%%%%%%%%%%%%%%%%%%%%%%%%%%%%%%%%%%%%%
 
%%%%%%%%%%%%%%%%%%%%%%%%%%%%%%%%%%%%%%%%%%%%%%%%%%%%%%%%%%%%%%%%%%%%%%%%%%%%%%%%%%%%%%%%%%%%%%%%%%%%%%%%%%%%%%%%%%%%%%%%%%%%%%%%%%%%%%%%%%%%%%%%%%%%%%%%%%%%%%%%%%%%%%%%%%%%%%%%%%%%%%%%%%%%%%%%%%%%%%%%%%%%%%%%%%%%%%%%%%%%%%%%%%%%%%%%%%%%

\noindent{\bf (iv)  Global existence:} Since the solution map defines a semiflow in $O_\alpha$, we have $t^+(u^0)=\infty$ whenever  the orbit $u([0,t^+(u^0));u^0)$ is relatively compact on $O_\alpha$.
This is part (iv) of Theorem~\ref{T1}.\\

\noindent{\bf (v)  Blow-up criterion:} Let $u^0\in O_\alpha$ with $t^+(u^0)<\infty$.
Assume that $u(\cdot;u^0):[0,t^+(u^0))\to E_\alpha$ is uniformly continuous but \eqref{1a} was not true. 
Then, the limit $\lim_{t\nearrow t^+(u^0)}u(t;u^0)$ exists in $O_\alpha,$ so that~$u([0,t^+(u^0));u^0)$ is relatively compact in $O_\alpha$, which contradicts (iv) of Theorem~\ref{T1}. 
This entails  $(a)$ from  Theorem~\ref{T1}~{ (v)}.

As for {part  $(b)$  of} Theorem~\ref{T1}~{ (v)}, let $E_1$ be compactly embedded in $E_0$. 
Assume for contradiction that~\eqref{1} was not true for some $\theta\in (\beta,1)$ with $(\xi-\theta)q<\min\{1,1+\gamma-\theta\}$. 
Without loss of generality we may assume that $\theta>\alpha$ (otherwise consider $\alpha_0\in (\beta,\theta)$ with $(\xi-\alpha_0)q<\min\{1,1+\gamma-\alpha_0\}$). 
Since then $E_{\theta}$
embeds compactly in $E_\alpha$, we may find a sequence $t_n\nearrow t^+(u^0)$ such that $(u(t_n))_n$ converges in $E_\alpha$ and its limit lies in $O_\alpha$. 
Using Proposition~\ref{P1} with $S_\alpha$ defined as the closure in $E_\alpha$ of the set $\{u(t_n)\,:\, n\in\N\}$, which is a compact subset of $O_\alpha$, we may extend the maximal solution. 
 This is a contradiction.\qed

%%%%%%%%%%%%%%%%%%%%%%%%%%%%%%%%%%%%%%%%%%%%%%%%%%%%%%%%%%%%%%%%%%%%%%%%%%%%%%%%%%%%%%%%%%%%%%%%%%%%%%%%%%%%%%%%%%%%%%%%%%%%%%%%%%%%%%%%%%%%%%%%%%%%%%%%%%%%%%%%%%%%%%%%%%%%%%%%%%%%%%%%%%%%%%%%%%%%%%%%%%%%%%%%%%%%%%%%%%%%%%%%%%%%%%%%%%%%

%%%%%%%%%%%%%%%%%%%%%%%%%%%%%%%%%%%%%%%%%%%%%%%%%%%%%%%%%%%%%%%%%%%%%%%%%%%%%%%%%%%%%%%%%%%%%%%%%%%%%%%%%%%%%%%%%%%%%%%%%%%%%%%%%%%%%%%%%%%%%%%%%%%%%%%%%%%%%%%%%%%%%%%%%%%%%%%%%%%%%%%%%%%%%%%%%%%%%%%%%%%%%%%%%%%%%%%%%%%%%%%%%%%%%%%%%%%% 

%%%%%%%%%%%%%%%%%%%%%%%%%%%%%%%%%%%%%%%%%%%%%%%%%%%%%%%%%%%%%%%%%%%%
%%%%%%%%%%%%%%%%%%%%%%%%%%%%%%%%%%%%%%%%%%%%%%%%%%%%%%%%%%%%%%%%%%%%
%%%%%%%%%%%%%%%%%%%%%%%%%%%%%%%%%%%%%%%%%%%%%%%%%%%%%%%%%%%%%%%%%%%%
\section{Proof of Theorem~\ref{XT:1}}\label{Sec2x}
%%%%%%%%%%%%%%%%%%%%%%%%%%%%%%%%%%%%%%%%%%%%%%%%%%%%%%%%%%%%%%%%%%%%
%%%%%%%%%%%%%%%%%%%%%%%%%%%%%%%%%%%%%%%%%%%%%%%%%%%%%%%%%%%%%%%%%%%%
%%%%%%%%%%%%%%%%%%%%%%%%%%%%%%%%%%%%%%%%%%%%%%%%%%%%%%%%%%%%%%%%%%%%

%%%%%%%%%%%%%%%%%%%%%%%%%%%%%%%%%%%%%%%%%%%%%%%%%%%%%%%%%%%%%%%%%%%%
%%%%%%%%%%%%%%%%%%%%%%%%%%%%%%%%%%%%%%%%%%%%%%%%%%%%%%%%%%%%%%%%%%%%
%%%%%%%%%%%%%%%%%%%%%%%%%%%%%%%%%%%%%%%%%%%%%%%%%%%%%%%%%%%%%%%%%%%%

 The proof of Theorem~\ref{XT:1} is similar to the proof of Theorem~\ref{T1}  with
  some modifications  which are necessary to adapt to the weaker assumptions on the nonlinearity $f$. The analogue of Proposition~\ref{P1} reads as follows:

\begin{prop}\label{XP1}
 Suppose~\eqref {V16} and let $R>0$.
Then, there exists~${T:=T(R)>0}$ such that, for each $u^0\in E_\alpha$ with $\|u^0\|_\alpha\le R$,
the problem \eqref{SC1SEEE} has a strong solution 
\begin{equation}\label{XSC1regul}
\begin{aligned}
u=u(\cdot;u^0)\in &\ C^1\big((0,T],E_0\big)\cap C\big((0,T],E_1\big)\cap C \big([0,T],E_\alpha\big) \\[1ex]
&\cap  C^{\min\{\alpha-\theta,\,{  (1-\mu q)_+}\}}\big([0,T], E_\theta\big)\cap  C_\mu\big((0,T],E_\xi\big)
\end{aligned}
\end{equation}
for any $\theta\in [0,\alpha]$ and $\mu>\xi-\alpha$. Moreover, there is a constant $c_0(R)>0$ such that
\begin{equation}\label{d3}
\| u(t;u^0)-u(t;u^1)\|_\alpha\le c_0(R)\, \| u^0-u^1\|_\alpha\,,\qquad 0\le t\le T\,,\quad u^0, u^1\in \bar{\mathbb{B}}_{E_\alpha}(0,R)\,.
\end{equation}
 Finally, if  $$\tilde u\in C^1\big((0,T],E_0\big)\cap C\big((0,T],E_1\big)\cap C  \big([0,T],E_{ \alpha}\big) \cap C_\nu\big((0,T],E_\xi\big)$$   with  $\nu\geq 0$  and $q\nu<\min\{ 1,1+\gamma-\alpha\}$
 is a solution to  \eqref{SC1SEEE} satisfying $\tilde u (0)=u^0 \in \bar{\mathbb{B}}_{E_\alpha}(0,R)$, then~${\tilde u=u(\cdot;u^0)}$.
\end{prop} 
\begin{proof} 

\noindent{\bf (i)} 
 Let  $U_{A}(t,s)$, $ 0\le s\le t$, be the evolution operator associated with $A\in C^\rho(\R^+,\mathcal{H}(E_1,E_0))$ and recall from \cite[II.Lemma~5.1.3]{LQPP} 
 that there exists  a constant $c>0$  such that
\begin{equation}\label{XSC1e6}
\|U_{A}(t,s)\|_{\mathcal{L}(E_\theta)} +(t-s)^{\theta-\vartheta_0}\|U_{A}(t,s)\|_{\mathcal{L}(E_\vartheta,E_\theta)}  \le c  \,, \quad 0\le s\le t\le 1\,,
\end{equation}
for $0\le \vartheta_0\le \vartheta\le \theta\le 1$ with $\vartheta_0<\vartheta$ if $0<\vartheta< \theta< 1$. 
 Recalling~\eqref{16b} and \eqref{Emb}, we may choose 
 a  positive constant $\mu$ such that 
\begin{equation}\label{Xconstants}
\xi-\alpha_0<\mu<\min\left\{\frac{1}{q},\frac{1+\gamma_0-\alpha}{q}\right\}\,,
\end{equation}
 for   appropriate $\alpha_0\in (0,\alpha)$ if $\alpha>0$, respectively  $\alpha_0:=0$ if $\alpha=0$ and, similarly,   $\gamma_0\in (0,\gamma)$ if $\gamma>0$, respectively~$\gamma_0:=0$ if~$\gamma=0$.
 
 We then define for $T \in (0,1)$ the Banach space
$$
\mathcal{W}_T:=C([0,T],E_\alpha) \cap C_\mu((0,T],E_\xi)\,.
$$
Given $u^0\in E_\alpha$ with $\|u^0\|_\alpha\le R,$ we set $$U_Au^0:=[t\mapsto U_A(t,0)u^0]$$ and deduce from \cite[II.Theorem~5.3.1]{LQPP}, \eqref{XSC1e6}, and \eqref{Xconstants}
that $U_Au^0\in \mathcal{W}_T$ with
$\|U_Au^0\|_{\mathcal{W}_T}\le c(R)$ for some $c(R)>0$.
Consequently, if $u\in \bar{\mathbb{B}}_{\mathcal{W}_T}\big(U_Au^0,1\big)$, then 
$$
\|u(t)\|_{ \alpha}+t^{\mu}\|u(t)\|_\xi \le c(R)\,,\quad t\in (0,T]\,,
$$
and it follows from~\eqref{Xa1d}  that, for $0<s,t\le T$,
\begin{align}\label{XSC1e7}
\begin{aligned}
\|f(u(t))-f(u(s))\|_{\gamma} &\le c(R) \big(t^{-\mu q}+s^{-\mu q}\big) \|u(t)-u(s)\|_{{ \alpha}}\\
&\quad + c(R) \big(t^{-\mu (q-1)}+s^{-\mu (q-1)}\big)\|u(t)-u(s)\|_{\xi} \,.
\end{aligned}
\end{align}
Similarly, we deduce from \eqref{Xa1d} that
\begin{align}\label{XSC1e8}
\|f(u(t))\|_{\gamma}&\le \|f(u(t))-f(0)\|_{\gamma} +\| f(0)\|_{\gamma}\le  c(R) t^{-\mu q}\,,\quad t\in (0,T]\,.
\end{align}
Also note for $u,v\in \bar{\mathbb{B}}_{\mathcal{W}_T}\big(U_Au^0,1\big)$ that
\begin{align}\label{XSC1e9}
\|f(u(t))-f(v(t))\|_{\gamma}\le c(R) t^{-\mu q} \|u(t)-v(t)\|_{{ \alpha}}+ c(R) t^{-\mu (q-1)} \|u(t)-v(t)\|_{\xi} \,,\quad t\in (0,T]\,.
\end{align}
Define now
\begin{equation}\label{XSC1Lambda}
F(u)(t):= U_{A}(t,0)u^0 +\int_0^t U_{A}(t,\tau) f(u(\tau))\,\rd \tau\,,\quad t\in [0,T]\,,\quad u\in \bar{\mathbb{B}}_{\mathcal{W}_T}\big(U_Au^0,1\big)\,.
\end{equation}
We claim that $F:\bar{\mathbb{B}}_{\mathcal{W}_T}\big(U_Au^0,1\big)\rightarrow \bar{\mathbb{B}}_{\mathcal{W}_T}\big(U_Au^0,1\big)$
 defines a contraction if $T=T(R)\in (0,1)$ is small enough.
\medskip

\noindent{\bf (ii)}
 Given  $u\in \bar{\mathbb{B}}_{\mathcal{W}_T}\big(U_Au^0,1\big)$ we first note, as in the proof of Proposition~\ref{P1}, 
 that $$u_\ve:=u(\cdot+\ve)\in C([0,T-\ve],E_\xi)\,,\quad f(u_\ve)\in C([0,T-\ve],E_0)$$ for every $\ve\in (0,T)$ so  that \cite[II.Theorem~5.3.1]{LQPP} and (the analogue of) \eqref{P1d}  yield
\begin{equation} \label{Xregdes2}
F(u)\in C\big((0,T],E_\theta\big)\,,\quad \theta\in (0,1)\,.
\end{equation}
 Analogously to~\eqref{o1} we may use~\eqref{XSC1e6}, \eqref{Xconstants},~\eqref{XSC1e8},  and \eqref{Xregdes2}  (noticing that $(\gamma,\xi)\not=(0,1)$)
 to obtain $F(u)\in C_\mu((0,T],E_\xi)$ with 
\begin{align}\label{est2}
\|F(u)-U_{A }u^0\|_{C_\mu((0,T],E_\xi)} &\le \frac{1}{2}
\end{align}
provided that $T=T(R)\in (0,1)$ is sufficiently small. Moreover, analogously to \eqref{gfg}, we deduce that~$F(u)\in C^{\min\{\alpha-\theta,1-\mu q\}}\big([0,T], E_\theta\big)$ for all $\theta\in[0,\alpha]$ and
\begin{equation}\label{est1}
\left\| F(u)- U_Au^0\right\|_{C([0,T],E_{ \alpha})}\le \frac{1}{2}, 
\end{equation}
by making $T=T(R)\in (0,1)$ smaller, if necessary. 
Gathering \eqref{est2} and \eqref{est1} we obtain that  the map~\mbox{$F:\bar{\mathbb{B}}_{\mathcal{W}_T}\big(U_Au^0,1\big) \to \bar{\mathbb{B}}_{\mathcal{W}_T}\big(U_Au^0,1\big)$}  is well-defined for~$T=T(R)\in (0,1)$   small enough. 
Furthermore, using \eqref{XSC1e6}, \eqref{XSC1e9},  and the  assumption $(\gamma,\xi)\not=(0,1)$,  we may show analogously to~\eqref{c111}  that 
\begin{align}
\|F(u)(t)-F(v)(t)\|_\theta 
&\le    c(R) \,d_{\mathcal{W}_T}(u,v)\,  \max\big\{t^{1+\gamma_0-\theta-\mu q}, t^{1-\mu q}\} ,\,\qquad t\in (0,T]  \,,\label{XSC1c111}
\end{align}
for $\theta\in \{\alpha,\,\xi\}$  and $u,v\in \bar{\mathbb{B}}_{\mathcal{W}_T}\big(U_Au^0,1\big)$. 
Recalling~\eqref{Xconstants}, we may choose~${T=T(R)\in (0,1)}$ sufficiently small to ensure that~
$F:\bar{\mathbb{B}}_{\mathcal{W}_T}\big(U_Au^0,1\big)\to \bar{\mathbb{B}}_{\mathcal{W}_T}\big(U_Au^0,1\big)$  is a contraction.
Thus,~$F$ has a unique fixed point $u=u(\cdot;u^0)\in \bar{\mathbb{B}}_{\mathcal{W}_T}\big(U_Au^0,1\big)$. 
 \medskip

\noindent{\bf (iii)}    In order to show that $u$ is a strong solution to \eqref{SC1SEEE} with regularity  \eqref{XSC1regul}, we handle the cases~$\gamma>0$ and~$\gamma=0$ separately.
 If $\gamma>0$, then $u\in C\big((0,T],E_1\big)\cap C^1\big((0,T],E_0\big)$ follows as in \eqref{regdes} from \cite[II.Theorem~1.2.2, II.Remarks~2.1.2 (e)]{LQPP}  and $u$ is thus a strong solution. 
 If~$\gamma=0$, then~${\xi<1}$ by assumption.
  We  consider again~$u_\ve:=u(\cdot+\ve)$ for~$\ve\in (0,T)$  and note that $u_\ve\in C([0,T-\ve],E_\theta)$ for each $\theta\in (0,1)$, see~\eqref{Xregdes2}. 
   Taking~$\theta\in (\xi,1)$, we then have $u_\ve(0)\in E_\theta$ and $f(u_\ve)\in C([0,T-\ve],E_0)$ so that \cite[II.Theorem~5.3.1]{LQPP}  together with   (the analogue of)~\eqref{P1d} 
   ensure   $u_\ve\in C^{\theta-\xi}([0,T-\ve],E_\xi)$. 
  Along with~\eqref{Xa1d} we conclude~$f(u_\ve)\in C^{\theta-\xi}([0,T-\ve],E_0)$. 
   Invoking now \cite[II.Theorem~1.2.1]{LQPP} we deduce $$u_\ve\in   C\big((0,T-\varepsilon],E_1\big)\cap C^1\big((0,T-\ve],E_0\big)$$
for each $\ve\in (0,T)$, hence $u$ is a strong solution to \eqref{SC1SEEE} enjoying the regularity properties~\eqref{XSC1regul} (for the constant $\mu$ fixed in Step 1).
That \eqref{XSC1regul} holds for every $\mu >\xi-\alpha$ follows   from the uniqueness property. 
\medskip

\noindent{\bf (iv)} To establish the Lipschitz continuity  with respect to the initial values, let~$u^0, u^1\in \bar{\mathbb{B}}_{E_\alpha}(0,R)$ and note for $\theta\in \{\alpha,\,\xi\}$ and $t\in (0,T]$ that
\begin{align*}
\|u(t;u^0)-u(t;u^1)\|_\theta\leq\|U_{A}(t,0)\|_{\mathcal{L}(E_\alpha,E_\theta)}\|u^0-u^1\|_{\alpha} 
 +\|F(u(\cdot;u^0))(t)-F(u(\cdot;u^1))(t)\|_\theta\,.
\end{align*}
Hence, taking  $\theta=\alpha$ and $\theta=\xi$  and using ~\eqref{XSC1e6} and~\eqref{XSC1c111}, we may make $T=T(R)\in (0,1)$ smaller, if necessary, to deduce that
$$
\|u(t;u^0)-u(t;u^1)\|_\alpha \le c_0(R)\,  \|u^0-u^1\|_{\alpha}\,,\quad u^0, u^1\in  \bar{\mathbb{B}}_{E_\alpha}(0,R)\,,\quad t\in [0,T]\,,
$$
for some constant $c_0(R)>0$.\medskip

\noindent{\bf (v)}  The uniqueness assertion is derived as in Proposition~\ref{P1}. 
\end{proof}

 The proof of Theorem~\ref{XT:1} now follows easily:

\subsection*{Proof of Theorem~\ref{XT:1}}

 The existence and uniqueness of a maximal strong solution to the Cauchy problem \eqref{SC1SEEE} 
 follows by standard arguments from   Proposition~\ref{XP1}.  Since $T=T(R)>0$ in Proposition~\ref{XP1} depends only on $R$, the blow-up criterion~\eqref{X1} readily follows when $t^+(u^0)<\infty$.
 Moreover, if $A(t)=A\in \mathcal{H}(E_1,E_0)$ for $t\geq 0$, then~\eqref{d3} implies as in the proof of Theorem~\ref{T1} that the  map~${(t,u^0)\mapsto u(t;u^0)}$ is a semiflow on~$E_\alpha$. 
 This then also ensures global existence for relatively compact orbits. $\qed$

\section{Basic Preliminaries}\label{Sec3}

In this section we collect some general results which will be used in the applications presented in  Section~\ref{Sec4} and Section~\ref{Sec6}.  By $W_{p}^{s}(\Omega)$ and $H_{p}^{s}(\Omega)$ for $s\in\R$ we denote the Sobolev-Slobodeckii spaces and the Bessel potential spaces, respectively \cite{Amann_Teubner,Tr78}.

\subsection{An Auxiliary Lemma}

The following auxiliary lemma about Nemitskii operators is in the spirit of \cite[Proposition~15.4]{Amann_ASNSP_84}  (see also \cite[Lemma~2.7]{WalkerEJAM}) and may be useful in certain applications when verifying assumption~\eqref{a1d} (see e.g. Section~\ref{Sec6}).

\begin{lem}\label{g-Lemma}
Let $n,d\in \N^*$ and let $\Omega$ be an open subset of $\R^n$. Consider $g\in C^1(\R^d,\R)$ with
\begin{equation}\label{key}
\vert \nabla g(r)-\nabla g(s)\vert \le c\big(1+\vert r\vert^{q-1}+\vert s\vert^{q-1}\big) \vert r-s\vert  \,,\quad r,s\in\R^d\,,
\end{equation}
for some constants $q\ge 1$ and $c>0$. 
Let $p\in [1,\infty)$ and $\mu\in (0,1)$.  Then $g(w)\in W_{p}^{\mu}(\Omega)$ for every~${w\in W_{p}^{\mu}(\Omega,\R^d)\cap L_\infty (\Omega,\R^d)}$.
 Moreover,
there is $K>0$ with
\begin{equation*}
    \begin{split}
\|g(w_1)-g(w_2)\|_{W_{p}^\mu}
     &\le  K \big(1+\| w_1\|_{\infty}^{q}+\| w_2\|_{\infty}^{q}\big) \| w_1-w_2\|_{W_{p}^\mu}\\
&\quad
+ K \big(1+\| w_1\|_{\infty}^{q-1}+\| w_2\|_{\infty}^{q-1}\big) \big(\| w_1\|_{W_{p}^\mu}+\| w_2\|_{W_{p}^\mu}\big) \| w_1-w_2\|_{\infty}
    \end{split}
    \end{equation*}
for  all $w_1,\, w_2\in W_{p}^{\mu}(\Omega,\R^d)\cap L_\infty (\Omega,\R^d)$.
\end{lem}

\begin{proof}
Note that \eqref{key} implies
\begin{equation}\label{key2x}
\vert \nabla g(r)\vert \le c_2\big(1+\vert r\vert^{q}\big) \,,\quad r\in\R^d\,.
\end{equation}
Let $w_1,\, w_2\in W_{p}^{\mu}(\Omega,\R^d)\cap L_\infty (\Omega,\R^d)$. 
It then follows from the fundamental theorem of calculus,~\eqref{key}, and \eqref{key2x} that, for~$x,\, y\in \Omega$,
    \begin{equation*}
    \begin{split}
    &\big\vert \big(g(w_1(x))-g(w_2(x))\big)-
    \big(g(w_1(y))-g(w_2(y))\big)\big\vert\\
    &\le \left\vert \int_0^1 \nabla g\big(w_1(x)+\tau [w_1(x)-w_2(x)]\big)\, \rd \tau\,
    \cdot\big(w_1(x)-w_2(x)-w_1(y)+w_2(y)\big)\right\vert\\
     &\ + \left\vert \int_0^1 \left\{\nabla g\big(w_1(x)+\tau
    [w_1(x)-w_2(x)]\big)-\nabla g\big(w_1(y)+\tau
    [w_1(y)-w_2(y)]\big)\right\}\, \rd \tau \cdot\left( w_1(y)-w_2(y)\right)\right\vert\\
    &\le c_3 \big(1+\|w_1\|_\infty^{q} +\|w_2\|_\infty^{q}\big) \,\big\vert w_1(x)-w_2(x)-w_1(y)+w_2(y)\big\vert \\
&\quad +c_3 \big( 1+\|w_1\|_\infty^{q-1}+\|w_2\|_\infty^{q-1} \big)\| w_1-w_2\|_\infty \,  \big( \vert w_1(x)-w_1(y) \vert +\big\vert w_2(x)-w_2(y)\big\vert \big)
\,.
  \end{split}
    \end{equation*}
The latter estimate together with \eqref{key2x} leads to
    \begin{equation*}
    \begin{split}
    \|g(&w_1)-g(w_2)\|_{W_{p}^\mu}^p\\
 &\le\, \| g(w_1)-g(w_2)\|_{L_p}^p
  + \int_{\Omega\times\Omega} \dfrac{\big\vert \big(g(w_1(x))-g(w_2(x))\big)-
    \big(g(w_1(y))-g(w_2(y))\big)\big\vert^p}{\vert
    x-y\vert^{n+\mu p}}\, \rd (x,y)\\
     &\le   c_4 \big(1+\| w_1\|_{\infty}^{q}+\| w_2\|_{\infty}^{q}\big)^p \| w_1-w_2\|_{W_p^\mu}^p\\
&\quad +c_4 \big( 1+\|w_1\|_\infty^{q-1}+\|w_2\|_\infty^{q-1} \big)^p\,  \big( \| w_1\|_{W_p^\mu} +\| w_2\|_{W_p^\mu}\big)^p\,\| w_1-w_2\|_\infty^p   \,,
    \end{split}
    \end{equation*}
as claimed.
\end{proof}

%%%%%%%%%%%%%%%%%%%%%%%%%%%%%
\subsection{Functional Analytic Setting for Applications}\label{Sec:42}
%%%%%%%%%%%%%%%%%%%%%%%%%%%%%

We provide the underlying function analytic setting for  the applications in the next section. 
In order to include Dirichlet and Neumann boundary conditions on an open, bounded, smooth subset $\Omega$ of $\R^n$  with $n\in\N^*$, we fix  $\delta\in\{0,1\}$ and define
$$
\mathcal{B}u:=u  \ \text{ on } \ \partial\Omega \ \text{ if } \ \delta=0\,,\qquad \mathcal{B}u:=\partial_\nu u\ \text{ on } \ \partial\Omega \ \text{ if } \ \delta=1\,,
$$
that is, $\delta=0$ refers to Dirichlet boundary conditions and $\delta=1$ refers to Neumann boundary conditions. For a  fixed $p\in (1,\infty)$ we introduce
$$
F_0:=L_p(\Omega)\,,\qquad F_1:= W_{p,\mathcal{B}}^{2}(\Omega)=H_{p,\mathcal{B}}^{2}(\Omega)=\{v\in H_{p}^{2}(\Omega)\,:\,  \mathcal{B} v=0 \text{ on } \partial\Omega\}\,, 
$$
and
$$
B_0:=\Delta_\mathcal{B}\in \mathcal{H}\big(W_{p,\mathcal{B}}^{2}(\Omega),L_p(\Omega)\big)\,.
$$
Let
$$
\big\{(F_\theta,B_\theta)\,:\, -1\le \theta<\infty\big\}\ 
$$
be the interpolation-extrapolation scale generated by $(F_0,B_0)$ and the 
complex interpolation functor~$[\cdot,\cdot]_\theta$ (see \cite[\S 6]{Amann_Teubner} and \cite[\S V.1]{LQPP}). That is,
\begin{equation}\label{f2x}
B_\theta\in \mathcal{H}(F_{1+\theta},F_\theta)\,,\quad -1\le \theta<\infty\,,
\end{equation}
  and, for $2\theta\neq -1-\delta+1/p$, we have (see \cite[Theorem~7.1; Equation (7.5)]{Amann_Teubner}\footnote{In fact, this is stated in \cite{Amann_Teubner} for
   $-2+\frac{1}{p}+\delta< 2\theta\le 2$. 
 Invoking then fact that $(1-\Delta_\mathcal{B})^{-1}\in \mathcal{L}(H_{p}^{2\theta-2}(\Omega),H_{p}^{2\theta}(\Omega))$ for $2<2\theta<2+\delta+1/p$, see \cite[Theorem 5.5.1]{Tr78}, we obtain the full range in \eqref{f2}.})
\begin{equation}\label{f2}
 F_\theta\doteq H_{p,\mathcal{B}}^{2\theta}(\Omega):=\left\{\begin{array}{ll} \{v\in H_{p}^{2\theta}(\Omega) \,:\, \mathcal{B} v=0 \text{ on } 
 \partial\Omega\}\,, &\delta+\frac{1}{p}<2\theta<2+\delta+\frac{1}{p} \,,\\[3pt]
	 H_{p}^{2\theta}(\Omega)\,, & -2+\frac{1}{p}+\delta< 2\theta<\delta +\frac{1}{p}\,.\end{array} \right.
\end{equation}
Also note from \cite[Remarks~6.1~(d)]{Amann_Teubner} (since $\Delta_\mathcal{B}$ has bounded imaginary powers) the reiteration property 
\begin{equation}\label{f3}
 [F_\alpha,F_\beta]_\theta\doteq F_{(1-\theta)\alpha+\theta\beta}
\end{equation}
and from \cite[Equations~(5.2)-(5.6)]{Amann_Teubner} the embeddings
 \begin{equation}\label{f4}
 H_{p}^{s}(\Omega)\hookrightarrow  W_{p}^{\tau}(\Omega)\hookrightarrow  H_{p}^{t}(\Omega)\,,\quad t<\tau<s\,.
\end{equation}
In the following, we use the letter $D$ and $N$ to indicate Dirichlet respectively Neumann boundary conditions (instead of $\mathcal{B}$).

\begin{rem}
 The reason for working in the Bessel potential scale $H_{p,\mathcal{B}}^{s}(\Omega)$ is that it is stable under complex interpolation, see~\eqref{f2}-\eqref{f3}. 
 However, using the \emph{almost reiteration property}  \cite[I.Remarks~2.11.2]{LQPP} (instead of~\eqref{f3}) and~\eqref{f4}, one may work just as well in the Sobolev scale~$W_{p,\mathcal{B}}^{s}(\Omega)$.
\end{rem}

Finally, we recall a useful tool on pointwise multiplication:

%%%%%%%%%%%%%%%%%%%%%%%%%%%%%%%%%%%%%%%%%%%%%%%%%%%%%%%%
\begin{prop}\label{PM}
Let $\Omega$ be an open, bounded, smooth subset of $\R^{n}$. Let $m\ge 2$ be an integer and let~${p,p_{j}\in [1,\infty)}$
 and $s,s_{j}\in (0,\infty )$ for $1\le j\le m$ be real numbers satisfying   $s < \min \{s_{j}\}$  along with $1/p\leq\sum_{j=1}^m1/p_j$ and
$$
s-\frac{n}{p} < \left\{ 
\begin{array}{ll}
\displaystyle{\sum\limits_{s_{j}<\frac{n}{p_{j}}}(s_{j}-\frac{n}{p_{j}})} &
\text{if} \quad \displaystyle{\min_{1\le j\le m}\left\{ s_{j}-\frac{n}{p_{j}} \right\} < 0} \ , \\
\displaystyle{\min_{1\le j \le m}\left\{ s_{j}-\frac{n}{p_{j}} \right\}} & \text{otherwise}\ .
\end{array}\right.
$$
Then pointwise multiplication
\begin{equation*}
\prod^{m}_{j=1}H_{p_{j}}^{s_{j}}(\Omega) \longrightarrow
H_{p}^{s}(\Omega )
\end{equation*}
is continuous.
\end{prop}
%%%%%%%%%%%%%%%%%%%%%%%%%%%%%%%%%%%%%%%%%%%%%%%%%%%%%%%%

\begin{proof}
 Proposition~\ref{PM} is a consequence of the more general result stated in \cite[Theorem~4.1]{AmannMult} (see also Remarks~4.2~(d) therein) and the embeddings~\eqref{f4} 
 (noticing  that the Sobolev spaces $W_p^s(\Omega)$ coincide with the Besov spaces $B_{p,p}^s(\Omega)$ for $s\in (0,\infty)\setminus\N$).
\end{proof}

%%%%%%%%%%%%%%%%%%%%%%%%%%%%%%%%%%%%%%%%%%%%%%%%
%%%%%%%%%%%%%%%%%%%%%%%%%%%%%%%%%%%%%%%%%%%%%%%%
%%%%%%%%%%%%%%%%%%%%%%%%%%%%%%%%%%%%%%%%%%%%%%%%
\section{Applications of Theorem~\ref{XT:1} to Chemotaxis Systems}\label{Sec4}
%%%%%%%%%%%%%%%%%%%%%%%%%%%%%%%%%%%%%%%%%%%%%%%%
%%%%%%%%%%%%%%%%%%%%%%%%%%%%%%%%%%%%%%%%%%%%%%%%
%%%%%%%%%%%%%%%%%%%%%%%%%%%%%%%%%%%%%%%%%%%%%%%%

We illustrate the findings of our abstract  result  from Theorem~\ref{XT:1} for the semilinear case  in the context of two chemotaxis systems, see \eqref{PP0} and \eqref{PP0xxx}.  
Exploiting the fact that we may choose~${\xi>\alpha}$ in Theorem~\ref{XT:1}, 
we prove local well-posedness for these chemotaxis systems in spaces of low regularity and obtain in this way quite general  global existence criteria.\\
 
In the following, let  $\Omega$  be an open, bounded,  smooth subset of $\R^n$  with $n\in\N^*$.

\subsection{Parabolic-Parabolic Equations}\label{Sec:51}

To begin with, we consider the cross-diffusion system
\begin{subequations}\label{PP0}
\begin{align}
\partial_t u&=\mathrm{div}\big(\nabla u-u\nabla v) +g(x,u,v)\,,&& t>0\,,\quad x\in \Omega\label{oi0}\,,\\
\partial_t v&=\Delta v +u-v\,,&& t>0\,,\quad x\in \Omega\,,\label{oi1}
\end{align}
where the nonlinearity $g$ is assumed to be of polynomial form
\begin{align}\label{polyn}
g(x,u,v)=\sum_{\ell=1}^M c_\ell(x) u^{m_\ell} v^{k_\ell}
\end{align}
with $m_\ell, k_\ell, M\in \N$ and sufficiently smooth functions $c_\ell$ (in fact, $c_\ell\in H_p^r(\Omega)$ with $r>n/p$). The evolution equations are
 subject to the initial conditions
\begin{align}\label{initcon}
u(0,x)=u^0(x)\,, \quad v(0,x)=v^0(x)\,,\qquad  x\in \Omega\,,
\end{align}
and the boundary conditions
\begin{align}\label{boundcon}
 \partial_\nu u=   \partial_\nu v&= 0 \quad  \text{on } \partial\Omega\,.
\end{align}
\end{subequations}

Even though \eqref{PP0} has a natural quasilinear structure, we may derive its well-posedness from Theorem~\ref{XT:1} by formulating \eqref{PP0} as a semilinear
evolution problem with different regularity and integrability assumptions for the variables $u$ and $v$. 
The advantage of this approach compared to the quasilinear theory is discussed in Remark~\ref{Rems:1} below.  
We also refer to \cite[Theorem~1]{Biler_AMSA_98} where  the existence of local weak solutions to \eqref{PP0} for initial data  $u_0\in L_p(\Omega)$ and $v_0\in H^1_p(\Omega)$ with $p>n$  
is established in the particular case $g=0$  and $n=2$.
In fact,  it is pointed out in \cite[Remarks]{Biler_AMSA_98} that one may weaken the restriction on $u_0$ and only assume that~$u_0\in L_{p/2}(\Omega)$.
As shown in the subsequent Theorem~\ref{T2'}, this is possible even in our setting of strong solutions and a function  $g$ 
 containing at most quadratic nonlinearities in~$u$, that is~$\max\{\, m_1,\ldots,\, m_\ell\}\leq 2$.
 In particular, in the physically relevant dimensions $n\in\{1,\, 2,\, 3\}$, we prove that already a priori estimates for $u$ in~$L_2(\Omega)$ would guarantee that the strong solution is globally defined.
In the case when~${ \max\{\, m_1,\ldots,\, m_\ell\}\geq  3}$, the  nonlinearity $g$ enforces more restrictive conditions on~$p$,  see Remark~\ref{Rems:1}~(a).

\begin{thm}\label{T2'}
Let $\max\{ \, m_1,\ldots,\, m_\ell\}\leq 2$,  $p\in (1,\infty) $ satisfy $p>n/2$, and choose $q\in (n,\infty)$  such that  
\[
  \frac{n}{p}-\frac{n}{q}<1\,.
\]
 Then   \eqref{PP0} generates a semiflow on $ L_p(\Omega)\times H_{q}^{1}(\Omega)$. 
Moreover, if the maximal strong solution 
$$
(u,v)\in C \big([0,t^+),L_p(\Omega)\times H_{q}^{1}(\Omega)\big)
$$   
to \eqref{PP0}, associated with an initial value $(u^0,v^0)\in L_p(\Omega)\times  H_{q}^{1}(\Omega)$, is not globally defined, that is, if~$t^+<\infty$, then
\[
\limsup_{t\nearrow t^+}\|u(t)\|_{L_p}=\infty\,.
\]
\end{thm}

 Before providing the proof of Theorem~\ref{T2'} we note:

\begin{rems}\label{Rems:1}
 {\bf (a)} Theorem~\ref{T2'} remains true if $m:=\max\{ m_1,\ldots,\, m_\ell\}\geq 3$ for the particular choice~$p=q>(m-1)n$. The proof of this result is similar to that of Theorem~\ref{T2'} and therefore omitted.\vspace{2mm}

{\bf (b)}   In order to allow for quite general initial data, we  shall consider equation \eqref{oi0} in a Bessel potential space of negative order. 
  Nevertheless,  a subsequent bootstrapping argument shows that~${(u,v)\in C^\infty((0,t^+)\times\bar\Omega, \R^2)}$  provided  that $c_\ell\in C^\infty(\bar\Omega)$   for  $1\leq \ell\leq M$.\vspace{2mm}

{\bf (c)}  We emphasize that   it is not at all clear whether the choice $(u^0,v^0)\in L_p(\Omega) \times H^{1}_q(\Omega)$ for the 
initial data in Theorem~\ref{T2'}  is possible when  using the quasilinear parabolic theory in \cite{Amann_Teubner} instead of Theorem~\ref{XT:1}, even if $g=0$ and $p=q>n$.
 Indeed, considering the evolution problem \eqref{PP0}  in the  ambient space~\mbox{$H_{p,N}^{2\sigma}(\Omega)\times H_{p,N}^{2\tau}(\Omega)$},
  the term $\mathrm{div}\big(\nabla u-u\nabla v)$ in \eqref{oi0} can be handled in several ways, 
  either quasilinear
\begin{align*}
&A_1(w)z:=\Delta  z_1- \mathrm{div}\big(z_1\nabla v)\,,\qquad A_2(w)z:=\Delta  z_1- \mathrm{div}\big(u\nabla z_2)\,,
\end{align*}
 for  $w=(u,v)$ and $z=(z_1,z_2)\in \mathrm{dom}(A_i(w)),$ or semilinear
\begin{align*}
& A_3z:=\Delta  z_1\ \text{ with }\ f(w):=\mathrm{div}\big(u\nabla v)
\end{align*}
for $z=(z_1,z_2)\in \mathrm{dom}(A_3)$. 
One then minimally requires that  $2\sigma,\, 2\tau >-1+1/p$ (to identify the extrapolation scale). 
In order to achieve  $A_1(w)z\in H_{p,N}^{2\sigma}(\Omega)$, one needs $z_1\nabla v\in H_{p}^{2\sigma+1}(\Omega)$, hence~${v\in H_{p}^{2\sigma+2}(\Omega)}$ with~${2\sigma+2>1,}$ 
so that a general $v^0\in H_p^1(\Omega)$ is not possible when using \cite{Amann_Teubner}.
Similarly,~$A_2(w)z\in H_{p,N}^{2\sigma}(\Omega)$ requires~${u\nabla z_2\in H_{p}^{2\sigma+1}(\Omega)}$ and hence~$u\in H_{p}^{2\sigma+1}(\Omega)$ with $2\sigma+1>0$ so that $u^0\in L_p(\Omega)$ seems impossible.
 Finally, for  $f(w)\in H_{p,N}^{2\sigma}(\Omega)$  one needs $u\nabla v\in H_{p}^{2\sigma+1}(\Omega)$ with~${2\sigma+1>0}$, and  hence neither~${u^0\in L_p(\Omega)}$ 
 nor $v^0\in H_p^1(\Omega)$   seems possible.
\end{rems}

We now establish the proof of Theorem~\ref{T2'}.

%%%%%%%%%%%%%%%%%%%%%%%%%%%%%%%%%%%%%%%%
\begin{proof}[Proof of Theorem~\ref{T2'}]
Let  $\ve$ be such that 
 \[
 0<2\ve< \min\Big\{1-\frac{1}{p},\,  1-\frac{n}{q},\, 1 -\frac{n}{p}+\frac{n}{q}\Big\}
 \]
and define
\begin{align*}
&E_0:= H_{p,N}^{-2\ve}(\Omega)\times H_{q,N}^{1-2\ve}(\Omega)\,, &E_1:= H_{p,N}^{2-2\ve}(\Omega)\times H_{q,N}^{3-2\ve}(\Omega) \,,
\end{align*}
so that
$$
E_\theta=H_{p,N}^{-2\ve+2\theta}(\Omega)\times H_{q,N}^{1-2\ve+2\theta}(\Omega)\,,\quad 2\theta\in [0,1]\setminus\{ 1+1/p+2\ve\,,\, 1/q+2\ve\}\,.
$$
Set
\[
0<\gamma:=\frac{\ve}{3}<\alpha:=\ve<\xi=\frac{1+\ve}{2}<1\,.
\]
It follows that
\begin{equation}\label{f5c}
2(\xi-\alpha)<  \min  \,\{1,\,1+\gamma-\alpha\}
\end{equation}
and, recalling ~\eqref{f2}-\eqref{f3}, we have
\[
E_\xi= H_{p,N}^{1- \ve}(\Omega)\times H_{q,N}^{2- \ve}(\Omega)\hookrightarrow E_\alpha= L_p(\Omega)\times H_{q}^{1}(\Omega)\hookrightarrow E_\gamma= H_{p,N}^{- 4\ve/3}(\Omega)\times H_{q,N}^{1- 4\ve/3}(\Omega)\,.
\]
Since  $H_{p,N}^{2-2\ve}(\Omega)\hookrightarrow H_{q,N}^{1-2\ve}(\Omega),$ we obtain from  \cite[I.~Theorem 1.6.1]{LQPP}  and \eqref{f2x}-\eqref{f2}   that 
$$
A:=\begin{pmatrix}
\Delta&0\\[1ex]
1&\Delta-1
\end{pmatrix}\in \mathcal{H}(E_1,E_0)\,.
$$
Defining $f:=(f_1+f_2,0)$ with
\[
\text{$f_1(w):=\big( -\mathrm{div}\big(u\nabla v),0\big)$ \qquad\text{and}\qquad $f_2(w):=\sum_{\ell=1}^M c_\ell(x) u^{m_\ell} v^{k_\ell}$}
\] 
for $w:=(u,v)$,
we may  thus recast \eqref{PP0} as a semilinear parabolic Cauchy problem
$$
w'=A w+f(w)\,,\quad t>0\,,\qquad w(0)=w^0:=(u^0,v^0)\,.
$$
We next show that  $f:E_\xi\to E_\gamma$ is well-defined and  that, given $R>0$, there is a constant $c(R)>0$ such that
\begin{equation}\label{condf}
\|f(w)-f(\bar w)\|_{E_\gamma}\le  c(R)\big[1+\|w \|_{E_\xi}+\|\bar w\|_{E_\xi}\big]\big[\big(1+\|w \|_{E_\xi}+\|\bar w\|_{E_\xi}\big) \|w -\bar w\|_{{E_\alpha}}+\|w-\bar w\|_{E_\xi}\big]
\end{equation}
for all  $w,\, \bar w\in E_\xi$ with $\|w\|_{E_\alpha},\, \|\bar w\|_{E_\alpha}\leq R.$
To this end we estimate, for each $w\in E_\xi$ with~${\|w\|_{E_\alpha}\leq R,}$
\[
\|f_1(w)\|_{H^{-4\ve/3}_{p,N}} \leq c\|u\nabla v\|_{H^{1-4\ve/3}_{p,N}}\leq c\|u\|_{H^{1-\ve}_{p,N}}\|v\|_{H^{2-\ve}_{q,N}}\leq c\|w\|_{E_\xi}^2\,,
\]
where the continuity of the multiplication 
\[
H_{p,N}^{1- \ve}(\Omega)\bullet H_{q,N}^{1- \ve}(\Omega)\longrightarrow H_{p,N}^{1- 4\ve/3}(\Omega)
\]
is used in the second step (note that $1-\ve>n/q)$, see Proposition~\ref{PM}. 
Moreover, since $H_{q,N}^{2- \ve}(\Omega)$ is an algebra with respect to pointwise multiplication, $\max\{ \, m_1,\ldots,\, m_\ell\}\leq 2$,  and since the multiplication 
\[
H_{p,N}^{1- \ve}(\Omega)\bullet H_{p,N}^{1- \ve}(\Omega)\longrightarrow L_p(\Omega)
\]
is continuous according to  Proposition~\ref{PM}, we have (assuming that $c_\ell\in H_p^r(\Omega)$ with $r>n/p$)
\[
\|f_2(w)\|_{H^{-4\ve/3}_{p,N}} \leq c \|f_2(w)\|_{L_p} \leq c(R)(1+\|u\|_{L_p}+\|u\|_{L_p}^2)  \leq c(R)(1+\|w\|_{E_\xi}^2)\,.
\]
This proves that  $f:E_\xi\to E_\gamma$  is well-defined.
Arguing as above, it is straightforward to show now that the local Lipschitz continuity property~\eqref{condf} is satisfied. 
We are thus in position to apply Theorem~\ref{XT:1} to deduce that \eqref{PP0} generates a semiflow in $E_\alpha=L_p(\Omega)\times H_{q}^{1}(\Omega)$.  

Let $w=(u,v) \in    C ([0,t^+),E_\alpha)$
be a maximal solution to \eqref{PP0} with $t^+=t^+(w^0)<\infty$. We then have
\begin{equation}\label{esti}
\limsup_{t\nearrow t^+}\|w(t)\|_{E_{\alpha}}=\limsup_{t\nearrow t^+}\|(u(t),v(t))\|_{L_p \times H_{q}^{1}}=\infty\,.
\end{equation}
Assume that  $u$ is bounded in $L_p(\Omega).$ We may also assume $w^0\in E_1,$ hence~${v_0\in H^2_{q,N}(\Omega)\hookrightarrow H^{2-2\ve}_{p,N}(\Omega)}.$
   It then follows from~\eqref{oi1} that $v$ is bounded in $H^{2-2\ve}_{p,N}(\Omega),$ hence also in $H^{1}_{q,N}(\Omega) $ due to the embedding~${H^{2-2\ve}_{p,N}(\Omega)\hookrightarrow H^{1}_{q,N}(\Omega)}$. 
   Hence, \eqref{esti} is equivalent to
\begin{equation*}
\limsup_{t\nearrow t^+}\|u(t)\|_{L_p}=\infty\,.
\end{equation*}
  This proves  the claim.
\end{proof}

\subsection{A Degenerate Chemotaxis System}\label{Sec:52}
We consider  the quasilinear evolution problem
\begin{subequations}\label{PP0xxx}
\begin{align}
\partial_t u&=\mathrm{div}(\nabla u-u\nabla w)\,,&& t>0\,,\quad x\in \Omega\label{iw}\,,\\
\partial_t v&=u-v\,,&& t>0\,,\quad x\in \Omega\,,\label{iiw}\\
\partial_t w&=\Delta w+v-w\,,&& t>0\,,\quad x\in \Omega\,,\label{iiiw}
\end{align} 
subject to the initial conditions
\begin{align}\label{initconxxx}
u(0,x)=u^0(x)\,,\quad v(0,x)=v^0(x)\,,\quad w(0,x)=w^0(x)\,,\qquad  x\in \Omega\,,
\end{align}
and the boundary conditions
\begin{align}\label{boundconxxx}
  (\nabla u-u\nabla w)\cdot \nu&= \partial_\nu w=0 \quad  \text{on } \partial\Omega\,.
\end{align}
\end{subequations}
This particular chemotaxis system is investigated in~\cite{PhL19} (see also \cite{HAm91} for a general strategy to handle this type of problems).
Although equations~\eqref{iw} and~\eqref{iiiw} are coupled through highest order terms, we proceed as in the previous Section~\ref{Sec:51} and 
 consider different regularities  and integrability properties for the variables~ $u,\, v,$ and~$w$
to reformulate \eqref{PP0xxx} as a semilinear parabolic evolution problem. 
 An application of   Theorem~\ref{XT:1} enables us to  impose  rather low regularities on the initial data and to derive sharp global existence criteria.

 \begin{thm}\label{T2xx}
Let $q\in (\max\{1,\, n-1\},\infty)$ and $1<p\le q$ with   
\begin{equation}\label{rt}
  \frac{n}{p}-\frac{n}{q}<1\,.
\end{equation} 
Let $s, \tau\in\R$ satisfy 
\begin{align*}
s\in \left(\max\left\{  \,-1+\frac{1}{p}\,,\,-1+\frac{n}{q}\right\},\frac{1}{q}\right)\,\quad\text{and}\,\quad
 \tau\in\left( \max\left\{\frac{n}{p}\,,\,\frac{n+1}{p}-\frac{n}{q}\right\}\,,\, s+2\right)\setminus\left\{1+\frac{1}{p}\right\}\,.
\end{align*} 
 Then \eqref{PP0xxx}   generates a semiflow on $H_{p,N}^s(\Omega)\times H_{p,N}^\tau(\Omega)\times H_{q,N}^{s+1}(\Omega)$. 
Moreover, if the maximal strong solution 
\[
(u,v,w)\in C \big([0,t^+),{H_{p,N}^s(\Omega)\times H_{p,N}^\tau(\Omega)\times H_{q,N}^{s+1}(\Omega)}\big)
\] 
to~\eqref{PP0xxx},  associated with an initial value
 $(u^0,v^0, w^0)\in H_{p,N}^s(\Omega)\times H_{p,N}^\tau(\Omega)\times H_{q,N}^{s+1}(\Omega)$,   is not globally defined,  then
\begin{equation}\label{g5y5}
\limsup_{t\nearrow t^+}\|u(t)\|_{H_{p}^s}=\infty\,.
\end{equation}
\end{thm}

Before establishing the proof of  Theorem~\ref{T2xx}, we  note:

\begin{rem}\label{Rems:22} 
 If $q>n$ and  $p\in (1,q]$ satisfy the condition~\eqref{rt}, then  one may choose~${s=0}$ (and an arbitrary~${\tau\in \big(\max\{n/p,(n+1)/p-n/q\}, 2\big)}$  such that $ \tau\neq 1+1/p$) 
 to obtain a semiflow on~${L_p(\Omega)\times H_{p}^\tau(\Omega)\times H_{q}^{1}(\Omega)}$. 
In this case, solutions are global provided 
\[
\text{$\sup_{t\in [0,t^+)\cap [0,T]}\|u(t)\|_{L_p}<\infty$\quad for each $T>0$\,.}
\]
 In particular, if $q=2p>n$, then one may choose $s=0$ and $\tau\in \big(\max\{n/p,(n+2)/(2p)\}, 2\big)$.   Hence,  a priori estimates for $u$ in~$L_2(\Omega)$
 ensure that the strong solution is globally defined in the physically relevant dimensions $n\in\{1,\, 2,\, 3\}$. For such estimates see \cite{PhL19}.
\end{rem}

We now establish  Theorem~\ref{T2xx}.

\begin{proof}[Proof of Theorem~\ref{T2xx}]
The assumptions on $s$ and $\tau$ along with~\eqref{rt} imply that we can choose a number $a$ such that
$$
\max\big\{\tau-2\,,\, -1+n/q\,,\,-1+1/p\big\}<a<\min\big\{\tau-1-n/p+n/q\,,\,s\big\}\,,
$$
that is, $a\in \big(\max\{-1+n/q,-1+1/p\},s\big)$ and
\begin{equation}\label{ww}
a+1+\frac{n}{p}-\frac{n}{q}<\tau<a+2\,.
\end{equation} 
 Set
\begin{equation}\label{g55}
 E_\theta:= H_{p,N}^{a+2\theta}(\Omega)\times H_{p,N}^{\tau}(\Omega) \times H_{q,N}^{a+1+2\theta}(\Omega)\,,\quad
 2\theta\in [0,2]\setminus\left\{ 1/q-a \,,\,  1+1/p-a \right\}\,.
\end{equation}
 We note that the middle component is independent of $\theta$ and that all spaces belong to the scale of~\eqref{f2}. We then have $E_\theta=[E_0,E_1]_\theta$. Set 
\begin{equation*}
2\alpha:=s-a\in (0,1)\setminus\left\{ 1/q-a \,,\,  1+1/p-a \right\}\,,\quad \gamma:=0\,,
\end{equation*}
and  choose 
\begin{equation*}
2\xi\in(1,1+\alpha)\setminus \left\{ 1/q-a,\, 1+1/p-a\right\}\,.
\end{equation*}
Note that $0=\gamma<\alpha<\xi<1$.
Since  $H_{p,N}^{a+2}(\Omega)\hookrightarrow H_{p,N}^{\tau}(\Omega)\hookrightarrow H_{q,N}^{a+1}(\Omega)$
 due to~\eqref{ww}, we obtain from  \cite[I.~Theorem 1.6.1]{LQPP}  and \eqref{f2x}-\eqref{f2}   that 
$$
A:=\begin{pmatrix}
\Delta&0&0\\[1ex]
1&-1&0\\[1ex]
0&1&\Delta-1
\end{pmatrix}\in \mathcal{H}(E_1,E_0)\,.
$$
Setting
$$
f(z):=\big(-\mathrm{div}(u\nabla w),0,0\big)\,,\quad z=(u,v,w)\,,
$$
we  may thus recast \eqref{PP0xxx} as the semilinear autonomous parabolic Cauchy problem
$$
z'=Az+f(z)\,,\quad t>0\,,\qquad z(0)=z^0:=(u^0,v^0,w^0)\,,
$$
in $E_0$.
Since $a+2\xi>a+1>n/q$, we derive continuity of the pointwise multiplication 
\begin{align*}
H_{p}^{a+2\xi}(\Omega)\bullet H_{q}^{a+2\xi}(\Omega)\hookrightarrow  H_{p}^{a+1}(\Omega)
\end{align*}
and therefore
\begin{align*}
\|\mathrm{div}\big(u\nabla w\big)\|_{H_{p}^{a}} \le  c\|u\nabla w\|_{H_{p}^{a+1}} \le c\|u\|_{H_{p}^{a+2\xi}}\|w\|_{H_{q}^{a+1+2\xi}}\,,\quad z=(u,v,w)\in E_\xi\,.
\end{align*}
We thus have
\begin{equation*}
\begin{split}
 \|f(z_1)-f(z_2)\|_{E_0}
\le c\big(\|z_1\|_{E_\xi} +\|z_2\|_{E_\xi} \big)\|z_1-z_2\|_{E_\xi}\,,\quad z_1,\, z_2\in E_\xi\,,
\end{split}
\end{equation*}
hence $f:E_\xi\to E_0$ satisfies \eqref{Xa1d} (with $q=2$ therein). 
 Since $2\xi<1+\alpha$, also  assumption  \eqref{16b} is satisfied and we infer from  Theorem~\ref{XT:1} that  \eqref{PP0xxx}  generates a semiflow on 
$$
 E_\alpha= H_{p,N}^{s}(\Omega)\times H_{p,N}^{\tau}(\Omega)\times H_{q,N}^{s+1}(\Omega)\,.
$$ 
In particular, for each $z^0=(u^0,v^0,w^0)\in H_{p,N}^{s}(\Omega)\times H_{p,N}^{\tau}(\Omega)\times H_{q,N}^{s+1}(\Omega)$ there is a unique strong solution~$z=(u,v,w)$ 
on some maximal interval $J=[0,t^+)$ and, if $t^+<\infty$, then
\begin{equation}\label{gl}
\limsup_{t\nearrow t^+}\|(u(t),v(t),w(t))\|_{H_{p}^s\times H_{p}^\tau\times H_{q}^{s+1}}=\infty\,.
\end{equation}

Finally, consider a solution $z=(u,v,w)$ on the maximal existence interval~${J=[0,t^+)}$ such that~$t^+<\infty$ and
$$
\|u(t)\|_{H_{p}^s}\le c_0<\infty\,,\quad t\in J\,.
$$
 We may assume without loss of generality that  $u^0\in H_{p,N}^{a+2}(\Omega)$ and  $w^0\in  H_{p,N}^{s+2}(\Omega)$ (as $s-a<1$).
Observing that~$\tau> s$, we infer from the bound on $u$ and \eqref{iiw} that
$$
\|v(t)\|_{H_{p}^s}\le c_1<\infty\,,\quad t\in J\,.
$$
In turn,  since $s>a$ and $w^0\in  H_{p,N}^{s+2}(\Omega)$ we  have
$w\in  BUC^{\vartheta}\big(J,H_{p,N}^{s+2-2\rho}(\Omega)\big)$ for some~${\rho>\vartheta>0}$ with 
\[
2\rho<\min\Big\{s-a,\, 1-\frac{n}{p}+\frac{n}{q}\Big\} 
\] from the latter estimate,~\eqref{iiiw}, and \cite[II.~Theorem~5.3.1]{LQPP}.
 Since $H_{p,N}^{s+2-2\rho}(\Omega)\hookrightarrow H_{q,N}^{s+1}(\Omega),$ we  get
$$
\|w(t)\|_{H_{q}^{s+1}}\le c_2<\infty\,,\quad t\in J\,.
$$
 In view of  $a+2>\tau>n/p$,   the pointwise multiplication 
\begin{align*}
 H_{p,N}^{\tau}(\Omega) \bullet H_{p,N}^{s-2\rho+1}(\Omega)\hookrightarrow  H_{p,N}^{a+1}(\Omega)
\end{align*}
is continuous and $ H_{p,N}^{a+2}(\Omega)\hookrightarrow H_{p,N}^{\tau}(\Omega)$. Setting
\[
B(t)y:=\Delta_N y-\mathrm{div}\big(y\nabla w(t)\big)\,,\quad y\in H_{p,N}^{a+2}(\Omega)\,,\quad t\in J\,,
\]
we may now deduce from \cite[I.Theorem~1.3.1, I.Corollary~1.3.2]{LQPP}   that 
\begin{align*}
B\in  BUC^{\vartheta}\big(J,\mathcal{L}(H_{p,N}^{a+2}(\Omega),H_{p,N}^{a}(\Omega))\big)\,,\quad
B(t)\in \mathcal{H}(H_{p,N}^{a+2}(\Omega),H_{p,N}^{a}(\Omega);\kappa,\omega)\,,\quad t\in J\,,
\end{align*}
for some $\kappa\ge 1$ and $\omega>0$. Since
$$
u'(t)= B(t)u(t)\,,\quad t\in J\,,\qquad u(0)=u^0\,,
$$
according to~\eqref{iw}, we conclude from \cite[II.~Theorem~5.4.1]{LQPP} that
$$
\|u(t)\|_{H_{p,N}^{a+2}}\le c_3\,,\quad t\in J\,.
$$
 Invoking again \eqref{iiw}  and recalling that $a+2>\tau$, we finally obtain
that
\begin{equation*}
\|(u(t),v(t),w(t))\|_{H_{p}^s\times H_{p}^\tau\times  H_{q}^{s+1}}\le c_4\,,\quad t\in J\,,
\end{equation*}
for some constant $c_4<\infty$, in contradiction to \eqref{gl}.
This ensures the global existence criterion~\eqref{g5y5}.
\end{proof}

\section{Application to Chemotaxis Systems with Density-Suppressed Motility}\label{Sec6}

We consider a model for autonomous periodic stripe pattern formation  (see \cite{JLZ22} and the literature therein)
\begin{subequations}\label{PP0xx}
\begin{align}
\partial_t u&=\Delta(u\chi(v))+ug(m)\,,&& t>0\,,\quad x\in \Omega\,,\label{u1}\\
\partial_t v&=\Delta v +u-v\,,&& t>0\,,\quad x\in \Omega\,,\label{u2}\\
\partial_t m&=\Delta m-ug(m)\,,&& t>0\,,\quad x\in \Omega\,,\label{u3}
\end{align}
  subject to the initial conditions
\begin{align}\label{initconxx}
u(0,x)=u^0(x)\,,\quad v(0,x)=v^0(x)\,,\quad m(0,x)=m^0(x)\,,\qquad  x\in \Omega\,,
\end{align}
and the boundary conditions
\begin{align}\label{boundconxx}
 \partial_\nu (u\chi(v))=   \partial_\nu v&= \partial_\nu m=0 \quad  \text{on } \partial\Omega\,.
\end{align}
\end{subequations}
 
The following  Theorem~\ref{T23x} is a consequence of Theorem~\ref{T1} 
 and establishes the existence and uniqueness of strong solutions only assuming~$u^0\in L_p(\Omega)$  with~$p>2n$ (in fact, as shown in the proof, even less is required), 
which does not seem to be derivable from \cite{Amann_Teubner}  directly. 
In the proof of Theorem~\ref{T23x} we take full advantage of the fact that we may choose~${\xi>\alpha}$ (in this context we actually have $\xi-\alpha>1/2$).
The global existence criterion~\eqref{b00} thus simplifies the one obtained  in \cite{JLZ22} for this system  being based on  a priori $L_\infty$-estimates for $u$. 

\begin{thm}\label{T23x}
Let $p\in (2n,\infty)$, $\chi\in C^{3-}(\R)$ with $\chi(r)\ge \chi_0>0$ for $r\in\R$, and $g\in C^{2-}(\R)$. 
Let further $\ve\in(0,1/p)$ and $\kappa>0$ satisfy
\begin{equation*} 
\frac{2n}{p}<2\kappa<1-5\ve
\end{equation*}
and set
 \begin{equation*} 
\bar\sigma:=-1+\frac{n}{p}+5\ve\,,\qquad  \bar\kappa:=2\kappa+4\ve\,,\qquad \bar\tau:=\frac{n}{p}+4\ve\,.
\end{equation*}
Then \eqref{PP0xx}   generates a semiflow on $H_{p,N}^{\bar\sigma}(\Omega)\times H_{p,N}^{\bar\kappa}(\Omega)\times H_{p,N}^{\bar\tau}(\Omega)$. 
Moreover, if a maximal strong solution 
\[(u,v,m)\in C \big([0,t^+),H_{p,N}^{\bar\sigma}(\Omega)\times H_{p,N}^{\bar\kappa}(\Omega)\times H_{p,N}^{\bar\tau}(\Omega)\big)\,
\] 
to~\eqref{PP0xx}   is such that  $t^+<\infty$,  then
\begin{equation}\label{b0}
\limsup_{t\nearrow t^+}\left\|(u(t),m(t))\right\|_{L_p \times H_{p}^{1}}=\infty\,.
\end{equation}
 In fact, if $g\ge 0$ and $(u^0,v^0,m^0)\in L_p(\Omega)\times H^1_{p,N}(\Omega)\times H_{p}^{1}(\Omega)$ are non-negative, then  $t^+<\infty$ implies
\begin{equation}\label{b00}
\limsup_{t\nearrow t^+}\left\|u(t) \right\|_{L_p}=\infty\,.
\end{equation}
\end{thm}

Before establishing the proof of  Theorem~\ref{T23x} we  note:

\begin{rem}\label{Rems:23}\phantom{a}
 The parameters are chosen such that  $L_p(\Omega)\times H_{p}^1(\Omega)\times H_{p}^{1}(\Omega)$ is contained in the space
   $H_{p,N}^{\bar\sigma}(\Omega)\times H_{p,N}^{\bar\kappa}(\Omega)\times H_{p,N}^{\bar\tau}(\Omega)$ of initial data.
\end{rem}

We now present the proof of Theorem~\ref{T23x}.

\begin{proof}[Proof of Theorem~\ref{T23x}]
For $w=(u,v,m)$ and $w^0=(u^0,v^0,m^0)$ we may recast \eqref{PP0xx} as a quasilinear Cauchy problem
\begin{equation}\label{QEP}
w'=A(w)w+f(w)\,,\quad t>0\,,\qquad w(0)=w^0\,,
\end{equation}
by  formally setting 
\begin{equation}\label{A}
A(w):=\big[(z_1,z_2,z_3)\mapsto \big(\mathrm{div}\big(\chi(v)\nabla z_1\big),(\Delta-1)z_2+z_1,\Delta z_3\big]
\end{equation}
and
\begin{equation}\label{f}
f(w):=\big(\mathrm{div}\big(u\chi'(v)\nabla v\big)+u g(m),0,-u g(m)\big)
\end{equation}
on suitable spaces which we introduce now. 
To this end, we set
\begin{equation*} 
2\sigma:=-1+\frac{n}{p}+\ve\,,\qquad  \alpha:=2\ve\,,\qquad 2\tau:=\frac{n}{p}\,,
\end{equation*}
and define
\begin{align*}
&E_0:=H_{p,N}^{2\sigma}(\Omega)\times H_{p,N}^{2\kappa}(\Omega)\times H_{p,N}^{2\tau}(\Omega)\,,\qquad E_1:=H_{p,N}^{2\sigma+2}(\Omega)\times H_{p,N}^{2\kappa+2}(\Omega)\times H_{p,N}^{2\tau+2}(\Omega)
\end{align*}
and note from \eqref{f2}-\eqref{f3}, since   $2\sigma\in (-1+1/p,0)$ and $2\kappa+2,\, 2\tau+2\in(2,3)$, that
\begin{align*}
&E_\theta:=[E_0,E_1]_\theta=H_{p,N}^{2\sigma+2\theta}(\Omega)\times H_{p,N}^{2\kappa+2\theta}(\Omega)\times H_{p,N}^{2\tau+2\theta}(\Omega)\,,\quad \theta\in [0,1]\setminus\Sigma\,,
\end{align*}
 where
\[
 \Sigma:=\Big\{\frac{1}{2}\Big(1+\frac{1}{p}-2\sigma\Big),\,\frac{1}{2}\Big(1+\frac{1}{p}-2\kappa\Big),\, \frac{1}{2}\Big(1+\frac{1}{p}-2\tau\Big)\Big\}.
\]
We further set
\begin{equation*}
  \gamma:=\ve\,,\qquad \beta:=\frac{3\ve}{2} \,,\qquad  \xi:=\frac{1}{2}+\frac{9\ve}{8}\,,
\end{equation*}
and note that $0<\gamma<\beta<\alpha<\xi<1$ and $\gamma,\, \beta,\,\alpha,\, \xi\not\in\Sigma $.
 We choose $\kappa'$ such that $2\kappa>2\kappa'>2n/p$ and note that, since
$$
2\kappa'+2\beta-\frac{n}{p}>\frac{n}{p}+  3\ve=:\rho\in(0,1)\,,
$$
we have~${H_p^{2\kappa'+2\beta}(\Omega)\hookrightarrow C^{\rho}(\bar\Omega)}$. 
 Moreover,   Lemma \ref{g-Lemma} and \eqref{f4}   imply
$$
[v\mapsto \chi(v)]\in C^{1-}\big(H_p^{2\kappa+2\beta}(\Omega),H_p^{2\kappa'+2\beta}(\Omega)\big)\,.
$$
Since $H_{p,N}^{2\sigma+1}(\Omega)$ is  an algebra with respect to pointwise multiplication 
and $2\kappa'+2\beta>2\sigma+1, $ we now infer from  \cite[Theorem~8.5]{Amann_Teubner}\footnote{Set $2\bar \alpha:=2\bar\beta:=2\sigma+2=1+n/p+\ve$. 
Then $1\leq 2\bar\alpha\leq 2$,  $2-2\bar\alpha\leq 2\bar\beta\leq2\bar\alpha$,  and  $2\bar\beta\neq 1+1/p$. 
The assumptions on~$\chi$ ensure that $(A_1(v),N) \in\mathcal{E}^{\bar\alpha}(\Omega) $ for each given   $v\in H_{p,N}^{2\kappa+2\beta}(\Omega)$ (in the notation of \cite{Amann_Teubner}) since~${\chi(v)\in C^\rho(\bar\Omega)}$ with $\rho>2\bar\alpha-1>0$. 
Applying  \cite[Theorem~8.5]{Amann_Teubner},
we get $A_1(v)\in \mathcal{H}\big(H_{p,N}^{2\sigma+2}(\Omega),H_{p,N}^{2\sigma}(\Omega)\big)$.} that for
\begin{equation*}
A_1(v)z_1:= \mathrm{div}\big(\chi(v)\nabla z_1\big) 
\end{equation*} 
we have
\begin{equation}\label{t4x}
A_1\in C^{1-}\big(H_p^{2\kappa+2\beta}(\Omega),\mathcal{H}\big(H_{p,N}^{2\sigma+2}(\Omega),H_{p,N}^{2\sigma}(\Omega)\big)\big)\,.
\end{equation} 
Now, 
we obtain from \eqref{t4x}, \eqref{f2x}-\eqref{f2} together with 
$$
H_{p,N}^{2\sigma+2}(\Omega)\hookrightarrow H_{p,N}^{1}(\Omega)\hookrightarrow H_{p,N}^{2\kappa}(\Omega)
$$
and a standard perturbation argument for the second component that, with $A$ defined in \eqref{A},
\begin{equation}\label{t4}
\big[w\mapsto A(w)\big]\in C^{1-}\big(E_\beta,\mathcal{H}(E_1,E_0)\big)\,.
\end{equation} 
As for the nonlinear part $f$, we first note from Proposition~\ref{PM}
that the pointwise multiplication (denoted by $\bullet$ in the following)
\begin{equation}\label{t6-}
H_p^{2\sigma+2\xi}(\Omega)\bullet H_p^{2\kappa'+2\beta}(\Omega)\bullet H_p^{2\kappa+2\xi-1}(\Omega)\longrightarrow H_p^{2\sigma+2\gamma+1}(\Omega)
\end{equation}
is continuous since $2\xi>2\gamma+1$, $ 2\sigma+2\gamma+1=n/p+3\ve>n/p$, and
\begin{align*} 2\kappa'+2\beta=2\kappa'+3\ve>\frac{2n}{p}+3\ve>\frac{n}{p}+3\ve ,\qquad 2\kappa+2\xi-1>\frac{2n}{p}+2\ve>\frac{n}{p}+3\ve \,.
\end{align*}
Consider in the following $w=(u,v,m)\in E_\xi$ with $\|w\|_{E_\beta}\le R$. Since
$$
[h\mapsto \chi'(h)]\in C^{1-}\big(H_p^{2\kappa+2\beta}(\Omega),H_p^{2\kappa'+2\beta}(\Omega)\big)
$$
is bounded on bounded sets according to Lemma \ref{g-Lemma} and \eqref{f4}, we deduce from \eqref{t6-} that
\begin{equation}\label{p1}
\left\|\mathrm{div}\big(u\chi'(v)\nabla v\big)\right\|_{H_p^{2\sigma+2\gamma}}\le c(R)\|u\|_{H_p^{2\sigma+2\xi}}\,\|v\|_{H_p^{2\kappa+2\xi}}\,.
\end{equation}
Noticing that also the pointwise multiplication
\begin{equation*}
H_p^{2\sigma+2\xi}(\Omega)\bullet H_p^{2\tau+\gamma+ \beta}(\Omega) \longrightarrow H_p^{2\tau+2\gamma}(\Omega)\hookrightarrow H_p^{2\sigma+2\gamma}(\Omega)
\end{equation*}
is continuous, since $2\tau>2\sigma$, $2\tau+\gamma+ \beta>2\tau+2\gamma>n/p$, and
$$
2\sigma+2\xi>\frac{n}{p}+2\ve=2\tau+2\gamma\,,
$$
and  the mapping
$$
[m\mapsto g(m)]\in C^{1-}\big(H_p^{2\tau+2\beta}(\Omega),H_p^{2\tau+\gamma+ \beta}(\Omega)\big)
$$
is bounded on bounded sets  (see again Lemma \ref{g-Lemma} and \eqref{f4}) we derive that
\begin{equation}\label{p2}
 \|u g(m)\|_{H_p^{2\sigma+2\gamma}}\le c  \|u g(m)\|_{H_p^{2\tau+2\gamma}}\le c(R)\|u\|_{H_p^{2\sigma+2\xi}} \,.
\end{equation}
It then follows from \eqref{p1}-\eqref{p2} and the definition of $f$ in \eqref{f} that
\begin{align}\label{t10}
\|f(w)-f(\bar w)\|_{E_\gamma}
\le c(R)\big[1+\|w\|_{E_\xi}+\|\bar w\|_{E_\xi}\big]\big[ \big(\|w\|_{E_\xi}+\|\bar w\|_{E_\xi}\big) \|w-\bar w\|_{E_\beta}+\|w-\bar w\|_{E_\xi}\big]
\end{align}
for $w, \bar w\in E_\xi$ with $\|w\|_{E_\beta}, \|\bar w\|_{E_\beta}\le R$.

Setting $q=2,$ it follows that the assumptions of Theorem~\ref{T1} are fulfilled in the context of the quasilinear evolution problem~\eqref{QEP}.
Consequently,  the solution map associated  with~\eqref{QEP} defines a semiflow on
\begin{align*}
E_\alpha=H_{p,N}^{2\sigma+2\alpha}(\Omega)\times H_{p,N}^{2\kappa+2\alpha}(\Omega)\times H_{p,N}^{2\tau+2\alpha}(\Omega) 
= H_{p,N}^{\bar\sigma}(\Omega)\times H_{p,N}^{\bar\kappa}(\Omega)\times H_{p,N}^{\bar\tau}(\Omega)\,.
\end{align*}
 Noticing that
$$
L_p(\Omega) \times H_{p}^{1}(\Omega)\times H_{p}^{1}(\Omega)\hookrightarrow E_\alpha
$$
since $\bar\sigma<0$ and $\bar\kappa, \bar\tau < 1$, we thus have 
\begin{equation}\label{b1}
\lim_{t\nearrow t^+}\left\|(u(t),v(t),m(t))\right\|_{L_p \times H_{p}^{1}\times H_{p}^{1}}=\infty
\end{equation}
for any  maximal strong solution $w:=(u,v,m)$ to~\eqref{QEP} on $J=[0,t^+)$  with $t^+<\infty$ corresponding to an initial value~${w^0=(u^0,v^0,m^0)\in E_\alpha}$.
We may assume that $w^0\in H_{p,N}^{1}(\Omega)\times H_{p,N}^{2 }(\Omega)\times H_{p,N}^{2 }(\Omega)$.
Then, if $\|u(t)\|_{L_p}\le c_0$ for $t\in J$,  we have $\|v(t)\|_{H_p^1}\le c_1$ for $t\in J$ according to \eqref{u3},  so that \eqref{b1} reduces to \eqref{b0}.

Finally, assume that $g\ge 0$ and let $(u^0,v^0,m^0)\in L_p(\Omega)\times H^1_{p,N}(\Omega)\times H_{p}^{1}(\Omega)$ satisfy $u^0,v^0,m^0\ge 0$ a.e. in $\Omega$. 
The comparison principle (together with a density argument and the semiflow property)  yields~${u(t),v(t),m(t) \ge 0}$ a.e. in~$ \Omega$ for all $t\in J$.
 We may assume,  via a bootstrapping argument,  that the initial values belong to~$H^3_{p,N}(\Omega)$. 
 Using again  the comparison principle   together with~\eqref{u3} we get~${\|m(t)\|_\infty\le \|m^0\|_\infty}$ for all $t\in J$. In particular, we obtain that~$\|g(m(t))\|_\infty\le c_2$ for~$t\in J$. 
 Assume now that $t^+<\infty$ and~${\|u(t)\|_{L_p}\le c_0}$ for~$t\in J$.  
Then $\|u(t)g(m(t))\|_{L_p}\le c_3$ for $t\in J$ and hence $\|m(t)\|_{H_p^1}\le c_4$ for~${t\in J}$ due to \eqref{u3}, in contradiction to \eqref{b0}. Consequently, we obtain~\eqref{b00} if~$t^+<\infty$.
\end{proof}

\bibliographystyle{siam}
\bibliography{Literature}

\begin{thebibliography}{10}

\bibitem{AT88}
{\sc P.~Acquistapace and B.~Terreni}, {\em On quasilinear parabolic systems},
  Math. Ann., 282 (1988), p.~315–335.

\bibitem{A83}
{\sc H.~Amann}, {\em {Gewöhnliche {D}ifferentialgleichungen}}, {de Gruyter
  Lehrbuch. [de Gruyter Textbook]}, Walter de Gruyter \& Co., Berlin, 1983.

\bibitem{Amann_ASNSP_84}
\leavevmode\vrule height 2pt depth -1.6pt width 23pt, {\em Existence and
  regularity for semilinear parabolic evolution equations}, Ann. Scuola Norm.
  Sup. Pisa Cl. Sci. (4), 11 (1984), pp.~593--676.

\bibitem{Am88}
\leavevmode\vrule height 2pt depth -1.6pt width 23pt, {\em {Dynamic theory of
  quasilinear parabolic equations. {I}. {A}bstract evolution equations}},
  Nonlinear Anal., 12 (1988), p.~895–919.

\bibitem{HAm91}
\leavevmode\vrule height 2pt depth -1.6pt width 23pt, {\em Highly degenerate
  quasilinear parabolic systems}, Ann. Scuola Norm. Sup. Pisa Cl. Sci. (4), 18
  (1991), pp.~135--166.

\bibitem{AmannMult}
\leavevmode\vrule height 2pt depth -1.6pt width 23pt, {\em Multiplication in
  {S}obolev and {B}esov spaces}, in Nonlinear analysis, Sc. Norm. Super. di
  Pisa Quaderni, Scuola Norm. Sup., Pisa, 1991, pp.~27--50.

\bibitem{Amann_Teubner}
\leavevmode\vrule height 2pt depth -1.6pt width 23pt, {\em {Nonhomogeneous
  linear and quasilinear elliptic and parabolic boundary value problems}}, in
  {Function spaces, differential operators and nonlinear analysis
  ({F}riedrichroda, 1992)}, vol.~133 of {Teubner-Texte Math.}, Teubner,
  Stuttgart, 1993, p.~9–126.

\bibitem{LQPP}
\leavevmode\vrule height 2pt depth -1.6pt width 23pt, {\em Linear and
  quasilinear parabolic problems. {V}ol. {I}}, vol.~89 of Monographs in
  Mathematics, Birkh\"{a}user Boston, Inc., Boston, MA, 1995.
\newblock Abstract linear theory.

\bibitem{A00}
\leavevmode\vrule height 2pt depth -1.6pt width 23pt, {\em On the strong
  solvability of the {N}avier-{S}tokes equations}, J. Math. Fluid Mech., 2
  (2000), pp.~16--98.

\bibitem{AW05}
{\sc H.~Amann and {\relax Ch}.~Walker}, {\em Local and global strong solutions
  to continuous coagulation-fragmentation equations with diffusion}, J.
  Differential Equations, 218 (2005), pp.~159--186.

\bibitem{An90}
{\sc S.~B. Angenent}, {\em {Nonlinear analytic semiflows}}, Proc. Roy. Soc.
  Edinburgh Sect. A, 115 (1990), pp.~91--107.

\bibitem{Biler_AMSA_98}
{\sc P.~Biler}, {\em Local and global solvability of some parabolic systems
  modelling chemotaxis}, Adv. Math. Sci. Appl., 8 (1998), pp.~715--743.

\bibitem{ClementLi}
{\sc P.~Cl\'{e}ment and S.~Li}, {\em Abstract parabolic quasilinear equations
  and application to a groundwater flow problem}, Adv. Math. Sci. Appl., 3
  (1993/94), pp.~17--32.

\bibitem{CS01}
{\sc {\relax Ph}.~Cl\'{e}ment and G.~Simonett}, {\em Maximal regularity in
  continuous interpolation spaces and quasilinear parabolic equations}, J.
  Evol. Equ., 1 (2001), pp.~39--67.

\bibitem{DaP96}
{\sc G.~{Da Prato}}, {\em Fully nonlinear equations by linearization and
  maximal regularity, and applications}, in Partial differential equations and
  functional analysis, vol.~22 of Progr. Nonlinear Differential Equations
  Appl., Birkhäuser Boston, Boston, MA, 1996, p.~80–92.

\bibitem{DaPG79}
{\sc G.~{Da Prato} and P.~Grisvard}, {\em Equations d'évolution abstraites non
  linéaires de type parabolique}, Ann. Mat. Pura Appl. (4), 120 (1979),
  p.~329–396.

\bibitem{DaPL88}
{\sc G.~{Da Prato} and A.~Lunardi}, {\em Stability, instability and center
  manifold theorem for fully nonlinear autonomous parabolic equations in
  {B}anach space}, Arch. Rational Mech. Anal., 101 (1988), p.~115–141.

\bibitem{G88}
{\sc D.~Guidetti}, {\em Convergence to a stationary state and stability for
  solutions of quasilinear parabolic equations}, Ann. Mat. Pura Appl. (4), 151
  (1988), p.~331–358.

\bibitem{JLZ22}
{\sc J.~Jiang, {\relax Ph}.~Lauren\c{c}ot, and Y.~Zhang}, {\em Global
  existence, uniform boundedness, and stabilization in a chemotaxis system with
  density-suppressed motility and nutrient consumption}, Comm. Partial
  Differential Equations, 47 (2022), pp.~1024--1069.

\bibitem{KPW10}
{\sc M.~K\"{o}hne, J.~Pr\"{u}ss, and M.~Wilke}, {\em On quasilinear parabolic
  evolution equations in weighted {$L_p$}-spaces}, J. Evol. Equ., 10 (2010),
  pp.~443--463.

\bibitem{PhL19}
{\sc {\relax Ph}.~Lauren\c{c}ot}, {\em Global bounded and unbounded solutions
  to a chemotaxis system with indirect signal production}, Discrete Contin.
  Dyn. Syst. Ser. B, 24 (2019), pp.~6419--6444.

\bibitem{LW22}
{\sc {\relax Ph}.~Lauren\c{c}ot and {\relax Ch}.~Walker}, {\em Well-posedness
  of the coagulation-fragmentation equation with size diffusion}, Differential
  Integral Equations, 35 (2022), pp.~211--240.

\bibitem{CPW10}
{\sc J.~LeCrone, J.~Pr\"{u}ss, and M.~Wilke}, {\em On quasilinear parabolic
  evolution equations in weighted {$L_p$}-spaces {II}}, J. Evol. Equ., 14
  (2014), pp.~509--533.

\bibitem{LeCroneSimonett20}
{\sc J.~LeCrone and G.~Simonett}, {\em On quasilinear parabolic equations and
  continuous maximal regularity}, Evol. Equ. Control Theory, 9 (2020),
  pp.~61--86.

\bibitem{Lu84}
{\sc A.~Lunardi}, {\em Abstract quasilinear parabolic equations}, Math. Ann.,
  267 (1984), p.~395–415.

\bibitem{Lu85}
\leavevmode\vrule height 2pt depth -1.6pt width 23pt, {\em Asymptotic
  exponential stability in quasilinear parabolic equations}, Nonlinear Anal., 9
  (1985), p.~563–586.

\bibitem{Lu85b}
\leavevmode\vrule height 2pt depth -1.6pt width 23pt, {\em Global solutions of
  abstract quasilinear parabolic equations}, J. Differential Equations, 58
  (1985), p.~228–242.

\bibitem{Lu87}
\leavevmode\vrule height 2pt depth -1.6pt width 23pt, {\em On the local
  dynamical system associated to a fully nonlinear abstract parabolic
  equation}, in Nonlinear analysis and applications ({A}rlington, {T}ex.,
  1986), vol.~109 of Lecture Notes in Pure and Appl. Math., Dekker, New York,
  1987, p.~319–326.

\bibitem{L95}
\leavevmode\vrule height 2pt depth -1.6pt width 23pt, {\em {Analytic Semigroups
  and Optimal Regularity in Parabolic Problems}}, {Progress in Nonlinear
  Differential Equations and their Applications, 16}, Birkhäuser Verlag,
  Basel, 1995.

\bibitem{MW_MOFM20}
{\sc B.-V. Matioc and {\relax Ch}.~Walker}, {\em On the principle of linearized
  stability in interpolation spaces for quasilinear evolution equations},
  Monatsh. Math., 191 (2020), pp.~615--634.

\bibitem{PW17}
{\sc J.~Pr\"{u}ss and M.~Wilke}, {\em Addendum to the paper ``{O}n quasilinear
  parabolic evolution equations in weighted {$L_p$}-spaces {II}''
  [{MR}3250797]}, J. Evol. Equ., 17 (2017), pp.~1381--1388.

\bibitem{P02}
{\sc J.~Prüss}, {\em Maximal regularity for evolution equations in
  {$L_p$}-spaces}, Conf. Semin. Mat. Univ. Bari,  (2002), p.~1–39 (2003).

\bibitem{PS16}
{\sc J.~Prüss and G.~Simonett}, {\em Moving interfaces and quasilinear
  parabolic evolution equations}, vol.~105 of Monographs in Mathematics,
  Birkhäuser/Springer, [Cham], 2016.

\bibitem{PSW18}
{\sc J.~Prüss, G.~Simonett, and M.~Wilke}, {\em Critical spaces for
  quasilinear parabolic evolution equations and applications}, J. Differential
  Equations, 264 (2018), p.~2028–2074.

\bibitem{Tr78}
{\sc H.~Triebel}, {\em {Interpolation Theory, Function Spaces, Differential
  Operators}}, North-Holland, Amsterdam, 1978.

\bibitem{WalkerEJAM}
{\sc {\relax Ch}.~Walker}, {\em Global existence for an age and spatially
  structured haptotaxis model with nonlinear age-boundary conditions}, European
  J. Appl. Math., 19 (2008), pp.~113--147.

\bibitem{WW06}
{\sc {\relax Ch}.~Walker and G.~F. Webb}, {\em Global existence of classical
  solutions for a haptotaxis model}, SIAM J. Math. Anal., 38 (2006/07),
  pp.~1694--1713.

\end{thebibliography}
\end{document}